\documentclass{amsart}

\usepackage{ fourier } 
\usepackage{ amssymb }
\usepackage{ booktabs }

\usepackage{ tikz }
\usetikzlibrary{ arrows, calc, positioning }
\newcommand{\Size}{6 mm}
\tikzset{ Square/.style = { 
inner sep = 0pt, minimum width = \Size, minimum height = \Size, draw=black, fill = none, align = center 
} }

\theoremstyle{definition}
\newtheorem{definition}{Definition}[section]

\newtheorem*{remark}{Remark}
\newtheorem*{problem}{Open problem}

\theoremstyle{plain}
\newtheorem{lemma}[definition]{Lemma}
\newtheorem{proposition}[definition]{Proposition}
\newtheorem{theorem}[definition]{Theorem}

\newcommand{\hor}{\circ}
\newcommand{\ver}{\bullet}

\newcommand{\liel}{[}
\newcommand{\lier}{]}
\newcommand{\jorl}{\{}
\newcommand{\jorr}{\}}

\newcommand{\sg}{\sigma}

\allowdisplaybreaks

\begin{document}

\title{LIE AND JORDAN PRODUCTS IN INTERCHANGE ALGEBRAS}

\author{Murray Bremner}

\address{Department of Mathematics and Statistics, University of Saskatchewan, Canada}

\email{bremner@math.usask.ca}

\author{Sara Madariaga}

\address{Department of Mathematics and Statistics, University of Saskatchewan, Canada}

\email{sara.madariaga@unirioja.es} 



\begin{abstract}
We study Lie brackets and Jordan products derived from associative operations $\hor, \ver$ 
satisfying the interchange identity 
$(w \ver x ) \hor ( y \ver z ) \equiv (w \hor y ) \ver ( x \hor z )$.
We use computational linear algebra, based on the representation theory of the symmetric group, 
to determine all polynomial identities of degree $\le 7$ relating 
(i) the two Lie brackets, (ii) one Lie bracket and one Jordan product, and (iii) the two Jordan products.
For the Lie-Lie case, there are two new identities in degree 6 and another two in degree 7.
For the Lie-Jordan case, there are no new identities in degree $\le 6$ and a complex set of
new identities in degree 7.
For the Jordan-Jordan case, there is one new identity in degree 4, two in degree 5,
and complex sets of new identities in degrees 6 and 7.
\end{abstract}

\maketitle

\thispagestyle{empty}


\section{Introduction}

\subsection*{Associative, Lie, and Jordan algebras}

An associative algebra is a vector space $A$ over a field $\mathbb{F}$ with a bilinear product 
$A \times A \to A$ denoted $(a,b) \mapsto ab$ satisfying associativity $(ab)c \equiv a(bc)$ 
where the symbol $\equiv$ means that the equation holds for all values of the indeterminates $a, b, c$.
If we replace the associative product $ab$ by the Lie bracket $[a,b] = ab - ba$, or the Jordan product
$\{a,b\} = ab + ba$, then we obtain a Lie algebra $A^-$, or a (special) Jordan algebra $A^+$.
Every polynomial identity satisfied by the Lie bracket in every associative algebra follows from 
anticommutativity and the Jacobi identity $[a,[b,c]] + [b,[c,a]] + [c,[a,b]] \equiv 0$.
However, there are `special' identities, occurring first in degree 8, satisfied by the Jordan product 
in every associative algebra which do not follow from commutativity and the Jordan identity 
$\{a,\{b,\{a,a\}\}\} \equiv \{\{a,b\},\{a,a\}\}$.

\subsection*{Two associative products}

Loday \cite{Z} describes of a number of algebraic structures with two associative 
products $\hor$, $\ver$ related by various polynomial identities:
\begin{enumerate}
\item
\emph{2-associative algebras}, in which the products satisfy no identities other than associativity; 
in this paper we use the name \emph{AA algebras}.
\item
\emph{Dual 2-associative algebras}, which satisfy these identities:
\[
a \hor ( b \ver c ) \equiv 0, \qquad
a \ver ( b \hor c ) \equiv 0, \qquad
( a \hor b ) \ver c \equiv 0, \qquad
( a \ver b ) \hor c \equiv 0.
\]
\item \label{dup}
\emph{Duplicial algebras} (or \emph{L-algebras}), which satisfy 
$( a \hor b ) \ver c \equiv a \hor ( b \ver c )$.
\item
\emph{Dual duplicial algebras}, which satisfy these identities:
\[
a \hor ( b \ver c ) \equiv ( a \hor b ) \ver c, \qquad
a \ver ( b \hor c ) \equiv 0, \qquad
( a \ver b ) \hor c \equiv 0.
\]
\item \label{As(2)}
\emph{As(2)-algebras}, which satisfy $( a \ast b ) \ast' c \equiv a \ast ( b \ast' c )$ for all 
$\ast, \ast' \in \{ \hor, \ver \}$; the corresponding nonsymmetric operad is self-dual.
\item \label{2com}
\emph{2-compatible algebras} (or \emph{As[2]-algebras}) in which any linear combination of the products 
is associative; this condition is equivalent to this identity:
\[
a \hor ( b \ver c ) + a \ver ( b \hor c )
\equiv
( a \hor b ) \ver c + ( a \ver b ) \hor c.
\]
\item \label{dual2com}
\emph{Dual 2-compatible algebras}, which satisfy the identities of items \eqref{As(2)} and \eqref{2com}.
\item
\emph{Diassociative algebras} (or \emph{associative dialgebras}), which satisfy the identity of duplicial
algebras \eqref{dup} together with
  \[
  a \ver ( b \hor c ) \equiv a \ver ( b \ver c ), \qquad
  ( a \hor b ) \hor c \equiv ( a \ver b ) \hor c.
  \]
The nonsymmetric operad dual to the operad for diassociative algebras corresponds to dendriform algebras 
which do not have associative products.
\item \label{inter}
\emph{Associative interchange algebras} (see Definition \ref{interchangedefinition} below) which are the 
main topic of this paper. 
The corresponding symmetric operad is generated by two binary operations which satisfy two quadratic 
relations (associativity) and one cubic relation (the interchange identity).
\end{enumerate}
Diassociative algebras, which arise as universal enveloping algebras of Leibniz algebras, have been studied 
by many authors; see Loday \cite{L}.
Otherwise, little work has been done on these structures: 
for 2-compatible algebras, see Dotsenko et al.~\cite{D,OS,St,ZBG1}; 
for duplicial algebras, see Bokut et al.~\cite{BCH,Le}.

\subsection*{Interchange identity}

The interchange identity has its origin in category theory and algebraic topology, 
related to the characterization of natural transformations; see Mac Lane \cite[\S XII.3]{ML}.
It plays a crucial role in the proof that the higher homotopy groups of a topological space 
are abelian, through the Eckmann-Hilton argument \cite{EH}.

\begin{definition} \label{interchangedefinition}
Loday and Vallette, \cite[\S 13.10.4]{LV}.
Let $\hor, \ver$ be binary operations on a set. 
The following multilinear polynomial identity is called the \textbf{interchange identity}:
  \begin{equation}
  \label{interchange}
  \tag{$\boxplus$}
  ( a \hor b ) \ver ( c \hor d ) \equiv ( a \ver c ) \hor ( b \ver d ).
  \end{equation}
An \textbf{interchange algebra} is a vector space with bilinear operations $\hor, \ver$ satisfying 
the interchange identity.
An \textbf{associative interchange algebra} is an interchange algebra in which the operations
are associative. 
\end{definition}

If we regard $\hor$ and $\ver$ as horizontal and vertical compositions respectively,
then \eqref{interchange} expresses the equivalence of two decompositions of a $2 \times 2$ array:
  \[
  ( a \hor b ) \ver ( c \hor d )
  \equiv \!\!
  \begin{array}{c}
  \begin{tikzpicture}[draw=black, x=\Size, y=\Size]
    \node [Square] at ($(0,0)$) {$a$};
    \node [Square] at ($(1,0)$) {$b$};
    \node [Square] at ($(0,-1.2)$) {$c$};
    \node [Square] at ($(1,-1.2)$) {$d$};
  \end{tikzpicture}
  \end{array}
  \!\! \equiv \!\!
  \begin{array}{c}
  \begin{tikzpicture}[draw=black, x=\Size, y=\Size]
    \node [Square] at ($(0,0)$) {$a$};
    \node [Square] at ($(1,0)$) {$b$};
    \node [Square] at ($(0,-1)$) {$c$};
    \node [Square] at ($(1,-1)$) {$d$};
  \end{tikzpicture}
  \end{array}
  \!\! \equiv \!\!
  \begin{array}{c}
  \begin{tikzpicture}[draw=black, x=\Size, y=\Size]
    \node [Square] at ($(0,0)$) {$a$};
    \node [Square] at ($(1.2,0)$) {$b$};
    \node [Square] at ($(0,-1)$) {$c$};
    \node [Square] at ($(1.2,-1)$) {$d$};
  \end{tikzpicture}
  \end{array}
  \!\! \equiv
  ( a \ver c ) \hor ( b \ver d ).
  \]
Kock \cite{K} found that associativity of the products in combination with the interchange identity 
produces unexpected commutativity phenomena in higher degrees: in degree 16, he gave an example of 
two equal monomials that have the same placement of parentheses and choice of operations but different 
permutations of the variables.
The present authors \cite{BM} recently used computer algebra to prove that 9 is the lowest degree in which such 
commutativity phenomena appear.
For related work, see DeWolf et al.~\cite{DW,E,PP}.

\subsection*{Polarization and depolarization}

A bilinear operation $a \ast b$ without symmetry (that is, neither commutative nor anti-commutative) can be 
polarized: decomposed into the sum of commutative and anti-commutative products $\{a,b\}$ and $[-,-]$:
  \begin{equation}
  \{ a, b \} = a \ast b + b \ast a,
  \qquad\qquad
  [ a, b ] = a \ast b - b \ast a.
  \end{equation}
The process can be reversed: depolarization combines commutative and anti-com\-mutative products to produce 
a new operation without symmetry:
  \begin{equation}
  a \ast b = \tfrac{1}{2} \big( \; \{ a, b \} + [a,b] \; \big).
  \end{equation}
This change of perspective allows us to study products without symmetry in terms of products with symmetry.
In particular, an associative operation can be decomposed into the corresponding Lie bracket and Jordan product. 
Associativity of the original operation is equivalent to the following identities relating the polarized 
products:
  \[
  [ a, \{ b, c \} ] \equiv \{ [ a, b ], c \} + \{ b, [ a, c ] \},
  \qquad\qquad
  [ b, [ a, c ] ] \equiv \{ \{ a, b \}, c \} - \{ a, \{ b, c \}\}.
  \]
These identities state that the Lie bracket is a derivation of the Jordan product, and that the (permuted) 
Lie triple product is the Jordan associator.
For further examples of polarization and depolarization, see Markl and Remm \cite{MR}.

Two associative products produce two Lie brackets and two Jordan products.
Hence we can study algebras with two associative products by finding the polynomial identities 
relating these Lie and Jordan structures. 
To go beyond the classical theory of Lie and Jordan algebras, we consider two Lie brackets, 
or two Jordan products, or one of each.

\subsection*{Outline of the paper}

We study multilinear polynomial identities relating the Lie and Jordan products obtained from 
polarization of the associative products in the free interchange algebra.
Our main computational tools are the Maple packages for rational and modular linear algebra:
LinearAlgebra and LinearAlgebra[Modular].
Our main theoretical tool is the representation theory of the symmetric group.

Section \ref{sectionpreliminaries} recalls basic definitions and constructions related to free 
associative and nonassociative algebras, introduces the notion of the expansion map which allows 
us to express polynomial identities as the kernel of a linear transformation, and summarizes 
what we require from the representation theory of the symmetric group.
In particular, we call a homogeneous polynomial identity of degree $n$ irreducible if its 
complete linearization generates a simple $S_n$-module.

Section \ref{sectionliejordan} considers one Lie bracket and one Jordan product.
There are no new identities in degrees $\le 6$: all such identities follow from anticommutativity 
and the Jacobi identity for the Lie bracket, and commutativity and the Jordan identity for the 
Jordan product.
There are 20 new irreducible identities in degree 7; we present explicitly the smallest 14 of
these identities in a compact nonlinear form.

Section \ref{sectionlielie} considers two Lie brackets.
There are no new identities in degrees $\le 5$, but there are two new irreducible identities in 
degree 6, and another two new irreducible identities in degree 7.
We present all these identities explicitly.

Sections \ref{sectionjordanjordan} considers two Jordan products.
There is a new irreducible identity in degree 4 which is a Jordan analogue of the interchange 
identity, two new irreducible identities in degree 5, and 14 new irreducible identities in degree 6.
We present all these identities explicitly in a compact nonlinear form.
There are 94 new irreducible identities in degree 7; we present explicitly the smallest 7 of these.

Section \ref{sectionJJ7proof} contains further details of computational methods used to obtain
the results on polynomial identities in degree 7 for two Jordan products.

\begin{remark}
The results are quite different in the three cases.
For two Lie brackets, the identities are simple enough to fit easily into the space restrictions 
of a journal article.
For one Lie bracket and one Jordan product, most but not all of the identities are simple enough 
to be publishable.
For two Jordan products, only a few of the identities in degree 7 are simple enough to be 
publishable, even in a compact nonlinear form.
The identities that we discovered in the Lie-Jordan and Jordan-Jordan cases, including those which 
do not appear explicitly in this paper, are available in ancillary files to the arXiv version;
see \cite{BMarXiv} for details.
\end{remark}

\begin{problem}
The question we address in this paper can equally well be asked of any class of two-associative 
algebras.
In particular, it would be of interest to determine the polynomial identities of degree $\le 7$ 
satisfied by the Lie brackets and Jordan products in the free algebras in the above-mentioned 
classes (1)--(8).
\end{problem}

\subsection*{Base field}

Unless otherwise noted, the base field $\mathbb{F}$ has characteristic 0.
In particular, this implies that the group algebra $\mathbb{F} S_n$ of the symmetric group $S_n$ 
(the left regular module) is semisimple for all $n \ge 1$, and that every homogeneous polynomial
identity of degree $n$ is equivalent to a finite set of multilinear identities \cite[Ch.~1]{ZSSS}.


\section{Preliminaries} \label{sectionpreliminaries}

\subsection*{Free $\Omega$-algebras}

Let $X = \{ x_1, x_2, \dots, x_n, \dots \}$ be a countable set of indeterminates, and let 
$\Omega = \{ \omega_1, \omega_2, \dots, \omega_n, \dots \}$ be a countable set of operation symbols
together with an arity function $\alpha\colon \Omega \to \mathbb{N} = \{ 1, 2, \dots, n, \dots \}$ 
indicating that $\omega_i$ represents an operation with $\alpha_i = \alpha(\omega_i)$ arguments.
We write $\Omega(X)$ for the set of \emph{monomials generated by $X$ using the operation $\Omega$}; 
thus $\Omega(X)$ is defined inductively by the conditions:
  \begin{enumerate}
  \item
  $X \subset \Omega(X)$.
  \item
  If $\omega_i \in \Omega$ and $m_1, \dots, m_{\alpha_i} \in \Omega(X)$ then 
  $\omega_i( m_1, \dots, m_{\alpha_i} ) \in \Omega(X)$.
  \end{enumerate}
Clearly $\Omega(X)$ is closed under the operations in $\Omega$.

We write $\mathbb{F}\{\Omega,X\}$ for the vector space over $\mathbb{F}$ with basis $\Omega(X)$;
if the operations are defined on basis monomials as in $\Omega(X)$ and extended bilinearly then 
this is the \emph{free $\Omega$-algebra} generated by $X$. 
By an \emph{$\Omega$-algebra over $\mathbb{F}$} we mean a vector space $A$ endowed with multilinear 
operations $\omega_i\colon A^{\alpha_i} \to A$ for each $\omega_i \in \Omega$ of arity 
$\alpha_i \in \mathbb{N}$; using $\omega_i$ for both the operation symbol and the multilinear 
operation should not cause confusion.
If $A$ is an $\Omega$-algebra over $\mathbb{F}$ then an element $f \in \mathbb{F}\{\Omega,X\}$ is 
called a \emph{polynomial identity} satisfied by $A$ if any substitution of elements of $A$ for 
the indeterminates in $f$ produces 0 when $f$ is evaluated by substituting the multilinear operations 
in $A$ for the corresponding operation symbols in $f$.
We write $f( x_1, \dots, x_n ) \equiv 0$ to indicate that the equation holds for all substitutions 
$x_1 = a_1$, \dots, $x_n = a_n$ of values $a_1, \dots, a_n \in A$.

Each term of a polynomial $f \in \mathbb{F}\{\Omega,X\}$ consists of a coefficient and a monomial, 
and each monomial consists of an \emph{association type}, which is a valid placement of operation 
symbols, together with an \emph{underlying sequence} of indeterminates.
Since $\mathbb{F}$ has characteristic 0, every polynomial identity
is equivalent to a finite set of (homogeneous) \emph{multilinear identities}.
A multilinear identity has a degree $n$ such that each monomial contains each of the $n$ indeterminates 
exactly once: the underlying sequence is a permutation of $x_1, \dots, x_n$.
We collect the terms of $f$ by their association types: if $f$ has $t$ different types, then 
$f = f^{(1)} + \cdots + f^{(t)}$ where $f^{(i)}$ includes the terms with association type $i$.
Each $f^{(i)}$ may be identified with an element of the group algebra $\mathbb{F} S_n$ since the 
monomials in $f^{(i)}$ differ only by a permutation of $x_1, \dots, x_n$.

We thus restrict our attention to multilinear identities, except when nonlinear (but still homogeneous)
identities allow us to write multilinear identities more compactly.
For a fixed degree $n$, we assume $X = \{ x_1, \dots, x_n \}$; then the symmetric group $S_n$ acts 
on multilinear polynomials by ignoring the association types and permuting the subscripts (not the 
positions) of the indeterminates in each monomial:
  \begin{equation}
  (\sg \cdot f)(x_1, \dots, x_n) = f(x_{\sg(1)}, \dots, x_{\sg(n)}).
  \end{equation}
This allows us to apply the representation theory of the symmetric group to the study of polynomial 
identities; for a detailed survey of this topic, see \cite{BMP}.

\subsection*{Free 2-associative algebras}

We will be concerned throughout with two bilinear associative products, $\Omega = \{ \hor, \ver \}$.
The association types for these products can be identified with the elements of $\Omega(X)$ where 
$X = \{ x \}$ has one element which we use as a placeholder for an argument to the operations.
The elements of $\Omega(X)$ are called \emph{AA types}; the number of AA types with $n$ operations
(corresponding to monomials of degree $n{+}1$) is the large Schr\"oder number, which is known to 
have this formula: 
  \begin{equation}
  aa(n) = \frac1n \sum_{k=1}^n 2^k \binom{n}{k} \binom{n}{k{-}1}.
  \end{equation}
See sequence A006318 in the OEIS: 1, 2, 6, 22, 90, 394, 1806, 8558, 41586, 206098, \dots.

Our normal form for AA types is obtained by inserting parentheses around every operation and
re-associating to the right where possible: $(- \hor -) \hor -$ becomes $- \hor (- \hor -)$, and 
similarly for $\ver$.
We need this convention to guarantee unique factorization in the presence of associativity.
Since operation $\hor$ is associative (and similarly for $\ver$), 
we clearly have the non-unique factorization $( a \hor b ) \hor c = a \hor ( b \hor c )$.
However, our convention excludes $( a \hor b ) \hor c$ and rewrites this monomial as $a \hor ( b \hor c )$.

Every normal AA form $t$ in degree $n$ has the unique factorization $t = t_1 \ast t_2$ into the 
product of normal forms $t_1, t_2$ of degrees $< n$ where $\ast \in \{ \hor, \ver \}$.
We assume that $\hor \prec \ver$, and define the \emph{deglex} total order on AA types as follows: 
We say $t \prec t'$ if and only if either 
  \begin{enumerate}
  \item
  $\deg(t) < \deg(t')$, or
  \item 
  $\deg(t) = \deg(t')$ and either 
    \begin{itemize}
    \item[]
    (a) $t_1 \prec t'_1$ or 
    (b) $t_1 = t'_1$ and $\ast \prec \ast'$ or 
    (c) $t_1 = t'_1$ and $\ast = \ast'$ and $t_2 \prec t'_2$.
    \end{itemize}
  \end{enumerate}
Since the associative operations $\hor, \ver$ have no symmetry, a multilinear AA monomial in degree $n$ 
is uniquely determined by an AA type and a permutation of the indeterminates $x_1, \dots, x_n$.
We order these multilinear AA monomials first by AA type then by lex order of the underlying permutation 
of the indeterminates.
Hence $\dim \mathrm{AA}(n) = aa(n{-}1) \; n!$ where $\mathrm{AA}(n)$ is the $S_n$-module in the 
AA operad: the multilinear subspace in degree $n$ of the free AA algebra on $n$ generators.

A multilinear polynomial identity $f( x_1, \dots, x_n ) \equiv 0$ of degree $n$ has $n{+}2$ consequences 
in degree $n{+}1$ for each operation $\ast \in \{ \hor, \ver \}$, obtained by substituting the
operation into the $n$ arguments of $f$ and substituting $f$ into the two arguments of the operation: 
  \begin{equation}
  \label{consequences}
  \left\{
  \begin{array}{l}
  f( x_1 \ast x_{n+1}, \dots, x_n ),
  \;\; \dots, \;\;
  f( x_1, \dots, x_i \ast x_{n+1}, \dots, x_n ),
  \;\; \dots, \;\;
  f( x_1, \dots, x_n \ast x_{n+1} ),
  \\
  f( x_1, \dots, x_i, \dots, x_n ) \ast x_{n+1},
  \quad
  x_{n+1} \ast f( x_1, \dots, x_i, \dots, x_n ).
  \end{array}
  \right.
  \end{equation}
The monomials in these consequences may not be in AA normal form, and hence may require rewriting.
Iteration produces consequences of $f$ in all higher degrees.

Every consequence of the interchange identity \eqref{interchange} in all degrees can be 
written as a difference of multilinear monomials $\iota(t) - \pi(t')$ where $\iota(t)$ has type $t$ 
and the identity permutation of the indeterminates, and $\pi(t')$ has type $t'$ and permutation $\pi$.
This allows us to reduce significantly the number of consequences in higher degrees.
We write $\mathrm{I\,I}(n)$ for the $S_n$-submodule of $\mathrm{AA}(n)$ generated by the consequences 
of \eqref{interchange}.
The quotient module $\mathrm{AA}(n) / \mathrm{I\,I}(n)$ is the $S_n$-module in the associative 
interchange operad.

\subsection*{Symmetric and skew-symmetric products}

We use the terms \emph{LJ types} (for one commutative and one anti-commutative operation), 
\emph{LL types} (for two anti-commutative operations), \emph{JJ types} (for two commutative operations).
In every degree $n$ the number $l\!j(n)$ of LJ/LL/JJ types is the same, and is sequence A226909
in the OEIS: 1, 2, 4, 14, 44, 164, 616, 2450, 9908, 41116, 173144, \dots.
In a general discussion we use LJ types.

We identify the LJ types with the elements of the subset $\mathrm{LJ} \subset \Omega(X)$ defined as 
follows where $X = \{ x \}$ and $\Omega = \{ \ast_1, \ast_2 \}$ contains two (anti-)commutative operations: 
  \begin{enumerate}
  \item
  $x \in \mathrm{LJ}$, and 
  \item
  if $v, w \in \mathrm{LJ}$ with $v \preceq w$ then $v \ast_1 w \in \mathrm{LJ}$ and $v \ast_2 w \in \mathrm{LJ}$
  where we use the deglex total order analogous to that on AA types with $\ast_1 \prec \ast_2$.
  \end{enumerate}
We write $\mathrm{LJ}(n)$ for the $S_n$-module in the LJ operad.
We have $\dim \mathrm{LJ}(n) < l\!j(n) \; n!$ for $n \ge 2$ since the (skew-)symmetries of the operations
imply that two monomials with the same LJ type but different permutations of the indeterminates can be equal 
up to a sign.

For any $\Omega$-algebra, the vector space consisting of all multilinear polynomial identities of 
degree $n$ satisfied by a given algebra (or given operations) is an $S_n$-module, which we call 
the module of \emph{all} identities.
Some of these multilinear identities in degree $n$ are consequences of those of lower degree, and hence 
do not provide any information; we call this the submodule of \emph{old} identities.
In each degree $n$, we are interested only in the identities which cannot be expressed in terms of 
known identities of lower degree.
These \emph{new} identities are elements of the quotient module of all by old identities.

For an algebra with two (skew-)symmetric products, a multilinear identity in degree $n$ generates 
only $n{+}1$ consequences in degree $n{+}1$ for each product, since the last two consequences in 
\eqref{consequences} are equal up to a sign.
In what follows, this process begins with the Jacobi and Jordan identities satisfied by the Lie
and Jordan products.

\subsection*{Expansion map}

The \emph{expansion map} $E_n\colon \mathrm{LJ}(n) \to \mathrm{AA}(n)$ is defined on basis monomials 
by replacing each of the (skew-)symmetric product symbols in the LJ monomials by the Lie bracket 
or Jordan product for the corresponding associative operation.
We write $\widetilde E_n= \pi \circ E_n$ for the composition of the expansion map with the natural 
surjection $\pi$:
  \begin{equation}
  \label{expansionmap}
  \widetilde E_n\colon \mathrm{LJ}(n) 
  \xrightarrow{\quad E_n \quad} \mathrm{AA}(n) 
  \xrightarrow{\quad \pi \quad} \mathrm{AA}(n) / \mathrm{I\,I}(n).
  \end{equation}
The multilinear identities in degree $n$ satisfied by the (skew-)symmetric products in the free 
associative interchange algebra are the nonzero elements of the kernel of $\widetilde E_n$.
Clearly $\widetilde E_n$ is an $S_n$-module morphism; we denote its kernel by $\mathrm{All}(n)$.

We compute a basis for $\mathrm{All}(n)$ as follows.
Given ordered monomial bases of $\mathrm{AA}(n)$ and $\mathrm{LJ}(n)$, we construct a block matrix
called the \emph{expansion matrix}:
  \begin{equation}
  \label{blockmatrix}
  B_n = 
  \begin{bmatrix} \Xi &\; O \\ X &\; I \end{bmatrix}
  \end{equation}
The entries of the expansion matrix are determined as follows:
\begin{itemize}
\item
The rows of $X$ contain the (coefficient vectors of the) expansions of the LJ monomials into the
free AA algebra so that $\mathrm{rowspace}(X) = \mathrm{image}(E_n)$.
\item
The rows of $\Xi$ contain the consequences of \eqref{interchange} so that 
$\mathrm{rowspace}(\Xi) = \mathrm{I\,I}(n)$.
\item
$O$ and $I$ are zero and identity matrices of the appropriate sizes.
\end{itemize}
We compute $\mathrm{RCF}(B_n)$ and identify the lower right block $K_n$ consisting of the nonzero rows 
whose leading 1s occur in columns $j > \dim \mathrm{AA}(n)$;
we have $\mathrm{rowspace}(K_n) = \mathrm{All}(n)$.
The consequences in degree $n$ of the known LJ identities of degree $< n$ generate a submodule 
$\mathrm{Old}(n) \subseteq \mathrm{All}(n)$; 
the new identities are $\mathrm{New}(n) = \mathrm{All}(n) / \mathrm{Old}(n)$.

The entries of $B_n$ belong to the set $\{ 0, \pm 1 \}$.
We prefer to compute the RCF over $\mathbb{Q}$, but this can be difficult if the matrix is large.
We often work over a finite field $\mathbb{F}_p$ where $p > n$; this guarantees that every finite-dimensional 
$S_n$-module is completely reducible and hence the structure constants $\{ 0, \pm 1 \}$ of the group algebra 
are formally the same as in characteristic 0 if we use symmetric representatives modulo $p$.

If the matrix is small enough, we can compute over $\mathbb{Z}$. 
We find the Hermite normal form (HNF) instead of the RCF and identify the lower right block as before.
We then apply the LLL algorithm for lattice basis reduction to compute a short integer basis for 
$\mathrm{All}(n)$, where `short' refers to the Euclidean lengths of the rows; see \cite{BP1} for details.

\subsection*{Representation theory}

As the degree of the identities increases, the numbers of AA and LJ monomials grow exponentially, 
and hence so does the matrix $B_n$ of equation \eqref{blockmatrix}.
The representation theory of $S_n$ allows us to replace $B_n$ with significantly smaller matrices.

The application of representation theory to the study of polynomial identities was initiated independently
by Malcev \cite{M} and Specht \cite{S} in 1950.
The computational implementation of these methods was pioneered by Hentzel \cite{H1,H2,H3} in the 1970s.
Important contributions were made by Clifton \cite{C} and Bondari \cite{B}.
For summaries of the theory with applications to polynomial identities, see \cite{BP3,BP4}.

We recall the decomposition of the group algebra $\mathbb{F}S_n$ into a direct sum of simple two-sided ideals
isomorphic to full matrix algebras.
The sum is over all partitions $\lambda$ of $n$, and $d(\lambda)$ is the number of standard tableaux for
the Young diagram of $\lambda$:
  \begin{equation}
  \label{Sndecomp}
  R \colon \mathbb{F}S_n \longrightarrow \bigoplus_\lambda M_{d(\lambda)}(\mathbb{F}).
  \end{equation}
Using representation theory we consider the partitions $\lambda$ one at a time.
To compute the projections $R_\lambda$ we used our own Maple implementation of Clifton's algorithm \cite{C}.
Theoretical and algorithmic details on computing the projections 
$R_\lambda\colon \mathbb{F}S_n \twoheadrightarrow M_{d(\lambda)}(\mathbb{F})$ 
and the inverse inclusions 
$M_{d(\lambda)}(\mathbb{F}) \hookrightarrow \mathbb{F}S_n$
are given in \cite{BMP}.

We write a multilinear polynomial $f \in \mathrm{LJ}(n)$ as the sum of $l\!j(n)$ components corresponding 
to the LJ types; in each component, the terms differ only by the permutation of the indeterminates.
Hence $f$ is an element of the direct sum of $l\!j(n)$ copies of $\mathbb{F} S_n$.
We fix a partition $\lambda$; to each component of $f$ we apply the surjective map $R_\lambda$ whose image 
is the corresponding matrix algebra in \eqref{Sndecomp}.
We combine horizontally the $l\!j(n)$ matrices of size $d(\lambda) \times d(\lambda)$ to obtain a matrix 
$R_\lambda(f)$ of size $d(\lambda) \times l\!j(n) d(\lambda)$: this is the representation matrix of $f$ 
for partition $\lambda$.
Each row of $\mathrm{RCF}(R_\lambda(f))$ generates a submodule of $(\mathbb{F}S_n)^{l\!j(n)}$ isomorphic 
to $[\lambda]$, the irreducible $S_n$-module for partition $\lambda$.
Hence $\mathrm{rank}\;R_\lambda(f)$ is the multiplicity of $[\lambda]$ in the submodule of 
$(\mathbb{F}S_n)^{l\!j(n)}$ generated by $f$.
In this way we compute generating sets for
$\mathrm{All}_\lambda(n)$, $\mathrm{Old}_\lambda(n)$, $\mathrm{New}_\lambda(n)$:
these are the isotypic components for $\lambda$ of the $S_n$-modules
$\mathrm{All}(n)$, $\mathrm{Old}(n)$, $\mathrm{New}(n)$.

Once we have identified the partitions which have new identities ($\mathrm{New}_\lambda(n) \ne \{0\}$)
we use nonlinear monomials 
to find compact forms of the identities.
For the partition $n = n_1 + \cdots + n_k$ where $n \ge n_1 \ge \cdots \ge 1$, we take $n_i$ copies of 
the indeterminate $a_i$ for $1 \le i \le k$.
The number of permutations of the indeterminates is no longer $n!$ but the much smaller multinomial 
coefficient $\binom{n}{n_1,\dots,n_k}$.
If the partition has a tail of length $t \ge 2$, meaning $n_{k+1-i} = 1$ for $1 \le i \le t$, then 
representation theory allows us to assume that every identity $f$ 
is an alternating function of the last $t$ indeterminates.

The following sections, and in particular section \ref{sectionJJ7proof}, give further explanation of 
the application of representation theory to polynomial identities 

\vfill


\section{Lie bracket and Jordan product} \label{sectionliejordan}

In this section we study the multilinear identities relating the Lie bracket of the horizontal 
operation $\hor$ and the Jordan product of the vertical operation $\ver$:
  \begin{equation} \label{liejordan}
  \liel a, b \lier = a \hor b - b \hor a,
  \qquad\qquad
  \jorl a, b \jorr = a \ver b + b \ver a.
  \end{equation}

\begin{proposition} \label{propositionLJ6}
Every multilinear identity in degree $\le 6$ satisfied by the Lie bracket and Jordan product 
\eqref{liejordan} in the free associative interchange algebra is a consequence of anti-commutativity 
and the Jacobi identity for the Lie bracket $\liel a, b \lier$ and commutativity and the (linearized) 
Jordan identity for the Jordan product $\jorl a, b \jorr$.
\end{proposition}

\begin{proof}
It suffices to prove that there are no new identities in degree 6, since existence of new identities 
in degree $n$ implies existence in degree $n{+}1$.
In degree 6, there are 394 AA types and 98 normalized consequences of the interchange identity;
there are 164 LJ types and 810 consequences of the Jacobi and Jordan identities.
For every partition $\lambda$ of 6 we construct the expansion matrix \eqref{blockmatrix},
compute its RCF, and extract the lower right block whose row space is the isotypic component for 
representation $[\lambda]$ of the $S_6$-module of multilinear identities relating the LJ monomials.
We also construct the matrix representing the consequences of the Jacobi and Jordan identities 
for partition $\lambda$, and compute its RCF.
We find that for every partition $\lambda$ of 6, the two RCFs are equal.
\end{proof}

\begin{lemma} \label{lemmaLJ7}
For every partition $\lambda$ of 7, the multiplicity of the simple $S_7$-module $[\lambda]$ in 
the modules $\mathrm{All}(7)$, $\mathrm{Old}(7)$, $\mathrm{New}(7)$ appears in Figure \ref{LJ7mults}.
Summing the multiplicities gives 20 new irreducible multilinear identities in degree 7,
relating the Lie bracket and Jordan product \eqref{liejordan} in the free associative interchange algebra, 
which are not consequences of the identities of lower degree (Proposition \ref{propositionLJ6}).
\end{lemma}

\begin{proof}
The methods are the same as in the proof of Proposition \ref{propositionLJ6} although the matrices are 
much larger.
The module $\mathrm{Old}(7)$ is always a submodule of $\mathrm{All}(7)$, 
but for some partitions $\lambda$ it is a proper submodule, and this indicates the existence of new identities.
Since we have the isomorphism of $S_7$-modules $\mathrm{New}(7) \cong \mathrm{All}(7) / \mathrm{Old}(7)$, 
the multiplicity of $[\lambda]$ in $\mathrm{New}(7)$ is the difference between its multiplicities in  
$\mathrm{All}(7)$ and $\mathrm{Old}(7)$.
\end{proof}

\begin{figure}[ht]
\small
\[
\boxed{
\begin{array}{c}
\begin{array}{c|ccccccc}
\lambda &\;\quad 
7 &\;\quad  61 &\;\quad  52 &\;\quad  51^2 &\;\quad  43 &\;\quad  421 &\;\quad  41^3 
\\
\mathrm{all} &\;\quad
536 &\;\quad 3234 &\;\quad 7525 &\;\quad 8140 &\;\quad 7521 &\;\quad 18957 &\;\quad 10953 
\\
\mathrm{old} &\;\quad
536 &\;\quad 3231 &\;\quad 7523 &\;\quad 8136 &\;\quad 7518 &\;\quad 18956 &\;\quad 10952 
\\
\mathrm{new} &\;\quad
0 &\;\quad 3 &\;\quad 2 &\;\quad 4 &\;\quad 3 &\;\quad 1 &\;\quad 1 
\end{array}
\\
\midrule
\begin{array}{cccccccc}
3^21 &\;\quad  32^2 &\;\quad  321^2 &\;\quad  31^4 &\;\quad  2^31 &\;\quad  2^21^3 &\;\quad  21^5 &\;\quad  1^7
\\
11378 &\;\quad 11399 &\;\quad 19160 &\;\quad 8314 &\;\quad 7677 &\;\quad 7766 &\;\quad 3375 &\;\quad 574
\\
11375 &\;\quad 11399 &\;\quad 19159 &\;\quad 8314 &\;\quad 7677 &\;\quad 7765 &\;\quad 3375 &\;\quad 573
\\
3 &\;\quad 0 &\;\quad 1 &\;\quad 0 &\;\quad 0 &\;\quad 1 &\;\quad 0 &\;\quad 1
\end{array}
\end{array}
}
\]
\caption{Multiplicities of new Lie-Jordan identities in degree 7}
\label{LJ7mults}
\end{figure}

\begin{theorem} \label{LJ7theorem}
Every multilinear identity in degree 7 satisfied by the Lie bracket and Jordan product \eqref{liejordan} 
in the free associative interchange algebra is a consequence of:
\begin{enumerate}
\item[(i)]
anti-commutativity and the Jacobi identity for the Lie bracket $\liel a, b \lier$;
\item[(ii)]
commutativity and the (linearized) Jordan identity for the Jordan product $\jorl a, b \jorr$;
\item[(iii)]
the linearizations of the identities of Figure \ref{liejordeg7identities} which are nonlinear 
forms of 14 of the 20 new irreducible identities identified by Lemma \ref{lemmaLJ7};
\item[(iv)]
the multilinear alternating sum identity of equation \eqref{LJdeg7lastrep};
\item[(v)]
the linearizations of the remaining 5 nonlinear identities which are available in an ancillary file 
for the \emph{\texttt{arXiv}} version of this paper; see \cite{BMarXiv} for details.
\end{enumerate}
\end{theorem}

\begin{proof}
Let $\lambda$ be one of the 10 partitions for which $\mathrm{New}_\lambda(7) \ne \{0\}$.
Let $\mathsf{all}(\lambda)$ and $\mathsf{old}(\lambda)$ be the matrices in RCF whose row spaces 
are $\mathrm{All}_\lambda(7)$ and $\mathrm{Old}_\lambda(7)$ respectively.
Let $\mathcal{A}(\lambda)$ and $\mathcal{O}(\lambda)$ be the sets of column indices corresponding 
to the leading 1s in $\mathsf{all}(\lambda)$ and $\mathsf{old}(\lambda)$.
A row of $\mathsf{all}(\lambda)$ whose leading 1 belongs to column 
$j \in \mathcal{A}(\lambda) \setminus \mathcal{O}(\lambda)$
represents a new identity.
For such a row, let $T$ be the set of LJ types $t$ for which the row has a nonzero entry in block $t$.
In many cases, $T$ is a very small subset of the LJ types.
This allows us to construct new identities without using representation theory.
For each type in $T$ we replace the positions of the indeterminates by all permutations of the multiset 
corresponding to $\lambda$ and normalize each monomial using the \mbox{(skew-)}symmetries of the LJ products.
For example, for $\lambda = 61$ we have only 7 permutations of the multiset $a^6b$ rather than $7!$
permutations of $abcdefg$.
The new identities for partition $\lambda$ are linear combinations of this relatively small set of 
nonlinear monomials.

\begin{figure}[ht]
\small
\[
\boxed{
\begin{array}{c|ccccccccc}
\lambda &\,\,\, 
61 &\,\,\, 52 &\,\,\, 51^2 &\,\,\, 43 &\,\,\, 421 &\,\,\, 
41^3 &\,\,\, 3^21 &\,\,\, 321^2 &\,\,\, 2^21^3
\\
\text{new} &\,\,\,
3 &\,\,\, 2 &\,\,\, 4 &\,\,\, 3 &\,\,\, 1 &\,\,\, 
1 &\,\,\, 3 &\,\,\, 1 &\,\,\, 1
\\
\text{terms} &\,\,\,
3{,}6{,}12 &\,\,\, 8{,}12 &\,\,\, 8{,}10{,}15{,}18 &\,\,\, 9{,}20{,}44 &\,\,\, 
20 &\,\,\, 12 &\,\,\, 15{,}24{,}40 &\,\,\, 21 &\,\,\, 24
\end{array}
}
\]
\caption{Number of terms in new Lie-Jordan identities in degree 7}
\label{LJ7terms}
\end{figure}

\medskip

We found new nonlinear identities for every $\lambda$ with $\mathrm{New}_\lambda(7) \ne \{0\}$.
Figure \ref{LJ7terms} displays the number of terms in each new identity for each $\lambda \ne 1^7$.
Figure \ref{liejordeg7identities} displays the new identities which have at most 20 terms: 
the 9 partitions in Figure \ref{LJ7terms} have respectively 3, 2, 4, 2, 1, 1, 1, 0, 0 identities
satisfying this condition.
We use power associativity of the Jordan product to write the monomials more compactly:
$\jorl a a \jorr = a^2$,
$\jorl a \jorl a a \jorr \jorr = a^3$,
$\jorl \jorl a a \jorr \jorl a a \jorr \jorr = a^4$.
Partition $\lambda = 1^7$ is a special case: the multilinear identities are linear combinations
of alternating sums over all 7 variables in the LJ types.
The new identity for that last partition is a single alternating sum:
  \begin{equation} 
  \label{LJdeg7lastrep}
  \sum_{\sg \in S_7}
  \epsilon(\sg) 
  \liel \jorl a^\sg, \liel b^\sg, c^\sg \lier \jorr, \jorl \liel d^\sg, e^\sg \lier, 
  \liel f^\sg, g^\sg \lier \jorr \lier
  \equiv 0.
  \end{equation}
Using skew-symmetry of the Lie bracket and symmetry of the Jordan product we can normalize 
the LJ monomials in this identity and collect terms with the same permutation of the variables 
to reduce the identity to a sum of only $7!/2^4 = 315$ terms.
\end{proof}

\begin{remark}
In the rest of this paper, we omit proofs when they are based on computations very similar to 
those already described.
\end{remark}

  \begin{figure}
  \small
  \[
  \boxed{
  \begin{array}{ll}
  %
  %
  61 &
     \jorl \liel a b \lier \liel a^2 a^3 \lier \jorr 
  -  \jorl \liel a a^2 \lier \liel b a^3 \lier \jorr 
  +  \jorl \liel b a^2 \lier \liel a a^3 \lier \jorr 
  \\ 
  \midrule
  %
  %
  &
     \jorl \liel a b \lier \liel a a^4 \lier \jorr 
  +2 \jorl \liel a a^2 \lier \liel a \jorl a \jorl a b \jorr \jorr \lier \jorr 
  -  \jorl \liel a a^2 \lier \liel a \jorl b a^2 \jorr \lier \jorr 
  -2 \jorl \liel a \jorl a b \jorr \lier \liel a a^3 \lier \jorr 
  +  \jorl \liel b a^2 \lier \liel a a^3 \lier \jorr 
  \\ 
  &
  +2 \jorl \liel a a^2 \lier \liel a^2 \jorl a b \jorr \lier \jorr 
  \\
  \midrule 
  %
  %
  &
     \liel a^2 \jorl a \jorl b \liel a a^2 \lier \jorr \jorr \lier 
  -  \liel a^2 \jorl b \jorl a \liel a a^2 \lier \jorr \jorr \lier 
  +  \liel a^2 \jorl a \liel b a^3 \lier \jorr \lier 
  +  \liel a^2 \jorl b \liel a a^3 \lier \jorr \lier 
  -2 \liel \jorl a b \jorr \jorl a \liel a a^3 \lier \jorr \lier 
  \\ 
  &
  -  \liel a^2 \jorl a^2 \liel b a^2 \lier \jorr \lier 
  -  \liel a^2 \jorl \jorl a b \jorr \liel a a^2 \lier \jorr \lier 
  +2 \liel \jorl a b \jorr \jorl a^2 \liel a a^2 \lier \jorr \lier 
  -  \liel a^3 \jorl a \liel b a^2 \lier \jorr \lier 
  -  \liel a^3 \jorl b \liel a a^2 \lier \jorr \lier 
  \\ 
  &
  +2 \liel \jorl b a^2 \jorr \jorl a \liel a a^2 \lier \jorr \lier 
  -  \liel a^3 \jorl \liel a b \lier a^2 \jorr \lier 
  \\ 
  \midrule
  %
  %
  52 &
     \jorl \liel a b \lier \liel b a^4 \lier \jorr 
  +2 \jorl \liel a b \lier \liel a^2 \jorl a \jorl a b \jorr \jorr \lier \jorr 
  -  \jorl \liel a b \lier \liel a^2 \jorl b a^2 \jorr \lier \jorr 
  -2 \jorl \liel a \jorl a b \jorr \lier \liel b a^3 \lier \jorr 
  \\ 
  &
  +2 \jorl \liel b a^2 \lier \liel a \jorl a \jorl a b \jorr \jorr \lier \jorr 
  -  \jorl \liel b a^2 \lier \liel a \jorl b a^2 \jorr \lier \jorr 
  +  \jorl \liel b a^2 \lier \liel b a^3 \lier \jorr 
  +2 \jorl \liel b a^2 \lier \liel a^2 \jorl a b \jorr \lier \jorr 
  \\ 
  \midrule
  %
  %
  &
     \liel a^2 \jorl a \jorl b \liel b a^2 \lier \jorr \jorr \lier 
  -  \liel a^2 \jorl b \jorl a \liel b a^2 \lier \jorr \jorr \lier 
  +  \liel a^2 \jorl b \liel b a^3 \lier \jorr \lier 
  -  \liel b^2 \jorl a \liel a a^3 \lier \jorr \lier 
  -  \liel a^2 \jorl \jorl a b \jorr \liel b a^2 \lier \jorr \lier 
  \\ 
  &
  +  \liel b^2 \jorl a^2 \liel a a^2 \lier \jorr \lier 
  +  \liel a^3 \jorl a \jorl b \liel a b \lier \jorr \jorr \lier 
  -  \liel a^3 \jorl b \jorl a \liel a b \lier \jorr \jorr \lier 
  -  \liel a^3 \jorl b \liel b a^2 \lier \jorr \lier 
  -  \liel \jorl a b^2 \jorr \jorl a \liel a a^2 \lier \jorr \lier 
  \\ 
  &
  +2 \liel \jorl b \jorl a b \jorr \jorr \jorl a \liel a a^2 \lier \jorr \lier 
  -  \liel a^3 \jorl \liel a b \lier \jorl a b \jorr \jorr \lier 
  \\ 
  \midrule
  %
  %
  51^2 &
     \jorl \liel a c \lier \liel b a^4 \lier \jorr 
  +2 \jorl \liel a b \lier \liel a^2 \jorl a \jorl a c \jorr \jorr \lier \jorr 
  -  \jorl \liel a b \lier \liel a^2 \jorl c a^2 \jorr \lier \jorr 
  -2 \jorl \liel a \jorl a c \jorr \lier \liel b a^3 \lier \jorr 
  \\ 
  &
  +2 \jorl \liel b a^2 \lier \liel a \jorl a \jorl a c \jorr \jorr \lier \jorr 
  -  \jorl \liel b a^2 \lier \liel a \jorl c a^2 \jorr \lier \jorr 
  +  \jorl \liel c a^2 \lier \liel b a^3 \lier \jorr 
  +2 \jorl \liel b a^2 \lier \liel a^2 \jorl a c \jorr \lier \jorr 
  \\ 
  \midrule
  %
  %
  &
   2 \liel a^3 \jorl a \jorl a \liel b c \lier \jorr \jorr \lier 
  -4 \liel \jorl a \jorl a b \jorr \jorr \jorl a \jorl a \liel a c \lier \jorr \jorr \lier 
  +4 \liel \jorl a \jorl a c \jorr \jorr \jorl a \jorl a \liel a b \lier \jorr \jorr \lier 
  +2 \liel \jorl b a^2 \jorr \jorl a \jorl a \liel a c \lier \jorr \jorr \lier 
  \\ 
  &
  -2 \liel \jorl c a^2 \jorr \jorl a \jorl a \liel a b \lier \jorr \jorr \lier 
  -  \liel a^3 \jorl \liel b c \lier a^2 \jorr \lier 
  +2 \liel \jorl a \jorl a b \jorr \jorr \jorl \liel a c \lier a^2 \jorr \lier 
  -2 \liel \jorl a \jorl a c \jorr \jorr \jorl \liel a b \lier a^2 \jorr \lier 
  \\ 
  &
  -  \liel \jorl b a^2 \jorr \jorl \liel a c \lier a^2 \jorr \lier 
  +  \liel \jorl c a^2 \jorr \jorl \liel a b \lier a^2 \jorr \lier 
  \\ 
  \midrule
  %
  %
  &
     \jorl \liel a c \lier \liel b a^4 \lier \jorr 
  -  \jorl \liel b c \lier \liel a a^4 \lier \jorr 
  -2 \jorl \liel a b \lier \liel a^2 \jorl c a^2 \jorr \lier \jorr 
  -2 \jorl \liel a c \lier \liel \jorl a b \jorr a^3 \lier \jorr 
  +  \jorl \liel b c \lier \liel a^2 a^3 \lier \jorr 
  \\ 
  &
  +2 \jorl \liel a a^2 \lier \liel b \jorl c a^2 \jorr \lier \jorr 
  +2 \jorl \liel a a^2 \lier \liel c \jorl a \jorl a b \jorr \jorr \lier \jorr 
  -  \jorl \liel a a^2 \lier \liel c \jorl b a^2 \jorr \lier \jorr 
  -2 \jorl \liel b a^2 \lier \liel a \jorl c a^2 \jorr \lier \jorr 
  \\ 
  &
  +  \jorl \liel c a^2 \lier \liel b a^3 \lier \jorr 
  -2 \jorl \liel c \jorl a b \jorr \lier \liel a a^3 \lier \jorr 
  -4 \jorl \liel a a^2 \lier \liel \jorl a b \jorr \jorl a c \jorr \lier \jorr 
  +4 \jorl \liel a \jorl a b \jorr \lier \liel a^2 \jorl a c \jorr \lier \jorr 
  \\ 
  &
  -4 \jorl \liel a \jorl a c \jorr \lier \liel a^2 \jorl a b \jorr \lier \jorr 
  +2 \jorl \liel c a^2 \lier \liel a^2 \jorl a b \jorr \lier \jorr 
  \\ 
  \midrule
  %
  %
  &
     \liel \jorl a b \jorr \jorl a \jorl c \liel a a^2 \lier \jorr \jorr \lier 
  -  \liel \jorl a b \jorr \jorl c \jorl a \liel a a^2 \lier \jorr \jorr \lier 
  -  \liel \jorl a c \jorr \jorl a \jorl b \liel a a^2 \lier \jorr \jorr \lier 
  +  \liel \jorl a c \jorr \jorl b \jorl a \liel a a^2 \lier \jorr \jorr \lier 
  \\ 
  &
  +  \liel \jorl a b \jorr \jorl a \liel c a^3 \lier \jorr \lier 
  +  \liel \jorl a b \jorr \jorl c \liel a a^3 \lier \jorr \lier 
  -  \liel \jorl a c \jorr \jorl a \liel b a^3 \lier \jorr \lier 
  -  \liel \jorl a c \jorr \jorl b \liel a a^3 \lier \jorr \lier 
  -  \liel \jorl a b \jorr \jorl a^2 \liel c a^2 \lier \jorr \lier 
  \\ 
  &
  -  \liel \jorl a b \jorr \jorl \jorl a c \jorr \liel a a^2 \lier \jorr \lier 
  +  \liel \jorl a c \jorr \jorl a^2 \liel b a^2 \lier \jorr \lier 
  +  \liel \jorl a c \jorr \jorl \jorl a b \jorr \liel a a^2 \lier \jorr \lier 
  -  \liel \jorl b a^2 \jorr \jorl a \liel c a^2 \lier \jorr \lier 
  \\ 
  &
  -  \liel \jorl b a^2 \jorr \jorl c \liel a a^2 \lier \jorr \lier 
  +  \liel \jorl c a^2 \jorr \jorl a \liel b a^2 \lier \jorr \lier 
  +  \liel \jorl c a^2 \jorr \jorl b \liel a a^2 \lier \jorr \lier 
  -  \liel \jorl b a^2 \jorr \jorl \liel a c \lier a^2 \jorr \lier 
  +  \liel \jorl c a^2 \jorr \jorl \liel a b \lier a^2 \jorr \lier 
  \\ 
  \midrule
  %
  %
  43 &
   2 \jorl \liel a b \lier \liel a \jorl a^2 b^2 \jorr \lier \jorr 
  -2 \jorl \liel a b \lier \liel a^2 \jorl a b^2 \jorr \lier \jorr 
  -  \jorl \liel a b \lier \liel b^2 a^3 \lier \jorr 
  +2 \jorl \liel a a^2 \lier \liel b \jorl a b^2 \jorr \lier \jorr 
  \\ 
  &
  -2 \jorl \liel a b^2 \lier \liel a \jorl b a^2 \jorr \lier \jorr 
  +  \jorl \liel a b^2 \lier \liel b a^3 \lier \jorr 
  -  \jorl \liel b b^2 \lier \liel a a^3 \lier \jorr 
  +2 \jorl \liel a a^2 \lier \liel \jorl a b \jorr b^2 \lier \jorr 
  +2 \jorl \liel a b^2 \lier \liel a^2 \jorl a b \jorr \lier \jorr 
  \\ 
  \midrule
  %
  %
  &
   2 \liel \jorl a b \jorr \jorl a \jorl b \liel b a^2 \lier \jorr \jorr \lier 
  -2 \liel \jorl a b \jorr \jorl b \jorl a \liel b a^2 \lier \jorr \jorr \lier 
  -  \liel b^2 \jorl a \jorl b \liel a a^2 \lier \jorr \jorr \lier 
  +  \liel b^2 \jorl b \jorl a \liel a a^2 \lier \jorr \jorr \lier 
  \\ 
  &
  +2 \liel \jorl a b \jorr \jorl b \liel b a^3 \lier \jorr \lier 
  -  \liel b^2 \jorl a \liel b a^3 \lier \jorr \lier 
  -  \liel b^2 \jorl b \liel a a^3 \lier \jorr \lier 
  -2 \liel \jorl a b \jorr \jorl \jorl a b \jorr \liel b a^2 \lier \jorr \lier 
  +  \liel b^2 \jorl a^2 \liel b a^2 \lier \jorr \lier 
  \\ 
  &
  +  \liel b^2 \jorl \jorl a b \jorr \liel a a^2 \lier \jorr \lier 
  +2 \liel \jorl b a^2 \jorr \jorl a \jorl b \liel a b \lier \jorr \jorr \lier 
  -2 \liel \jorl b a^2 \jorr \jorl b \jorl a \liel a b \lier \jorr \jorr \lier 
  -  \liel \jorl a b^2 \jorr \jorl a \liel b a^2 \lier \jorr \lier 
  \\ 
  &
  -  \liel \jorl a b^2 \jorr \jorl b \liel a a^2 \lier \jorr \lier 
  -2 \liel \jorl b a^2 \jorr \jorl b \liel b a^2 \lier \jorr \lier 
  +2 \liel \jorl b \jorl a b \jorr \jorr \jorl a \liel b a^2 \lier \jorr \lier 
  +2 \liel \jorl b \jorl a b \jorr \jorr \jorl b \liel a a^2 \lier \jorr \lier 
  \\ 
  &
  -  \liel \jorl a b^2 \jorr \jorl \liel a b \lier a^2 \jorr \lier 
  -2 \liel \jorl b a^2 \jorr \jorl \liel a b \lier \jorl a b \jorr \jorr \lier 
  +2 \liel \jorl b \jorl a b \jorr \jorr \jorl \liel a b \lier a^2 \jorr \lier 
  \\ 
  \midrule
  %
  %
  %
  421 &
   2 \liel \jorl a c \jorr \jorl a \jorl b \liel b a^2 \lier \jorr \jorr \lier 
  -2 \liel \jorl a c \jorr \jorl b \jorl a \liel b a^2 \lier \jorr \jorr \lier 
  -  \liel b^2 \jorl a \jorl c \liel a a^2 \lier \jorr \jorr \lier 
  +  \liel b^2 \jorl c \jorl a \liel a a^2 \lier \jorr \jorr \lier 
  \\ 
  &
  +2 \liel \jorl a c \jorr \jorl b \liel b a^3 \lier \jorr \lier 
  -  \liel b^2 \jorl a \liel c a^3 \lier \jorr \lier 
  -  \liel b^2 \jorl c \liel a a^3 \lier \jorr \lier 
  -2 \liel \jorl a c \jorr \jorl \jorl a b \jorr \liel b a^2 \lier \jorr \lier 
  +  \liel b^2 \jorl a^2 \liel c a^2 \lier \jorr \lier 
  \\ 
  &
  +  \liel b^2 \jorl \jorl a c \jorr \liel a a^2 \lier \jorr \lier 
  +2 \liel \jorl c a^2 \jorr \jorl a \jorl b \liel a b \lier \jorr \jorr \lier 
  -2 \liel \jorl c a^2 \jorr \jorl b \jorl a \liel a b \lier \jorr \jorr \lier 
  -  \liel \jorl a b^2 \jorr \jorl a \liel c a^2 \lier \jorr \lier 
  \\ 
  &
  -  \liel \jorl a b^2 \jorr \jorl c \liel a a^2 \lier \jorr \lier 
  +2 \liel \jorl b \jorl a b \jorr \jorr \jorl a \liel c a^2 \lier \jorr \lier 
  +2 \liel \jorl b \jorl a b \jorr \jorr \jorl c \liel a a^2 \lier \jorr \lier 
  -2 \liel \jorl c a^2 \jorr \jorl b \liel b a^2 \lier \jorr \lier 
  \\ 
  &
  -  \liel \jorl a b^2 \jorr \jorl \liel a c \lier a^2 \jorr \lier 
  +2 \liel \jorl b \jorl a b \jorr \jorr \jorl \liel a c \lier a^2 \jorr \lier 
  -2 \liel \jorl c a^2 \jorr \jorl \liel a b \lier \jorl a b \jorr \jorr \lier 
  \\ 
  \midrule
  %
  %
  41^3 &
   4 \liel \jorl a \jorl a b \jorr \jorr \jorl a \jorl a \liel c d \lier \jorr \jorr \lier 
  -4 \liel \jorl a \jorl a c \jorr \jorr \jorl a \jorl a \liel b d \lier \jorr \jorr \lier 
  +4 \liel \jorl a \jorl a d \jorr \jorr \jorl a \jorl a \liel b c \lier \jorr \jorr \lier 
  -2 \liel \jorl b a^2 \jorr \jorl a \jorl a \liel c d \lier \jorr \jorr \lier 
  \\ 
  &
  +2 \liel \jorl c a^2 \jorr \jorl a \jorl a \liel b d \lier \jorr \jorr \lier 
  -2 \liel \jorl d a^2 \jorr \jorl a \jorl a \liel b c \lier \jorr \jorr \lier 
  -2 \liel \jorl a \jorl a b \jorr \jorr \jorl \liel c d \lier a^2 \jorr \lier 
  +2 \liel \jorl a \jorl a c \jorr \jorr \jorl \liel b d \lier a^2 \jorr \lier 
  \\ 
  &
  -2 \liel \jorl a \jorl a d \jorr \jorr \jorl \liel b c \lier a^2 \jorr \lier 
  +  \liel \jorl b a^2 \jorr \jorl \liel c d \lier a^2 \jorr \lier 
  -  \liel \jorl c a^2 \jorr \jorl \liel b d \lier a^2 \jorr \lier 
  +  \liel \jorl d a^2 \jorr \jorl \liel b c \lier a^2 \jorr \lier 
  \\ 
  \midrule
  %
  %
  3^21 &
     \liel \jorl a b^2 \jorr \jorl a \jorl b \liel a c \lier \jorr \jorr \lier 
  +  \liel \jorl a b^2 \jorr \jorl b \jorl a \liel a c \lier \jorr \jorr \lier 
  -  \liel \jorl a \jorl b c \jorr \jorr \jorl a \jorl b \liel a b \lier \jorr \jorr \lier 
  -  \liel \jorl a \jorl b c \jorr \jorr \jorl b \jorl a \liel a b \lier \jorr \jorr \lier 
  \\ 
  &
  -  \liel \jorl b a^2 \jorr \jorl a \jorl b \liel b c \lier \jorr \jorr \lier 
  -  \liel \jorl b a^2 \jorr \jorl b \jorl a \liel b c \lier \jorr \jorr \lier 
  -  \liel \jorl b \jorl a c \jorr \jorr \jorl a \jorl b \liel a b \lier \jorr \jorr \lier 
  -  \liel \jorl b \jorl a c \jorr \jorr \jorl b \jorl a \liel a b \lier \jorr \jorr \lier 
  \\ 
  &
  +  \liel \jorl c \jorl a b \jorr \jorr \jorl a \jorl b \liel a b \lier \jorr \jorr \lier 
  +  \liel \jorl c \jorl a b \jorr \jorr \jorl b \jorl a \liel a b \lier \jorr \jorr \lier 
  -  \liel \jorl a b^2 \jorr \jorl \liel a c \lier \jorl a b \jorr \jorr \lier 
  +  \liel \jorl a \jorl b c \jorr \jorr \jorl \liel a b \lier \jorl a b \jorr \jorr \lier 
  \\ 
  &
  +  \liel \jorl b a^2 \jorr \jorl \liel b c \lier \jorl a b \jorr \jorr \lier 
  +  \liel \jorl b \jorl a c \jorr \jorr \jorl \liel a b \lier \jorl a b \jorr \jorr \lier 
  -  \liel \jorl c \jorl a b \jorr \jorr \jorl \liel a b \lier \jorl a b \jorr \jorr \lier 
  %
  %
  %
  %
  %
  \end{array}
  }
  \]
  \caption{14 of the 20 new nonlinear Lie-Jordan identities in degree 7}
  \label{liejordeg7identities}
  \end{figure}


\section{Two Lie brackets} \label{sectionlielie}

In this section we study the multilinear identities relating the two Lie brackets:
  \begin{equation} \label{2lie}
  \liel a, b \lier_\hor = a \hor b - b \hor a,
  \qquad\qquad
  \liel a, b \lier_\ver = a \ver b - b \ver a.
  \end{equation}

\begin{proposition} \label{propositionLL5}
In degree $\le 5$, every multilinear identity relating the Lie brackets \eqref{2lie}
in the free associative interchange algebra is a consequence of anticommutativity and the Jacobi 
identity for the two Lie brackets.
\end{proposition}

\begin{lemma}
For every partition $\lambda$ of 6, the multiplicity of the simple $S_6$-module $[\lambda]$ in 
the modules $\mathrm{All}(6)$, $\mathrm{Old}(6)$, $\mathrm{New}(6)$ appears in Figure \ref{LL6mults}.
There are two new irreducible multilinear identities in degree 6, relating the two Lie brackets 
\eqref{2lie} in the free associative interchange algebra, which are not consequences of the identities 
of lower degree (Proposition \ref{propositionLL5}).
Both of these new identities occur for partition $\lambda = 21^4$.
\end{lemma}

\begin{figure}[ht]
\small
\[
\boxed{
\begin{array}{cccccccccccc}
\lambda &\; 
6 &\;  51 &\;  42 &\;  41^2 &\;  33  &\; 321 &\;  31^3 &\;  2^3 &\;  2^21^2 &\;  21^4 &\;  1^6
\\
\midrule
\mathrm{all} &\;
164 &\; 788 &\; 1364 &\; 1496 &\; 734 &\; 2308 &\; 1422 &\; 704 &\; 1236 &\; 676 &\; 130
\\
\mathrm{old} &\;
164 &\; 788 &\; 1364 &\; 1496 &\; 734 &\; 2308 &\; 1422 &\; 704 &\; 1236 &\; 674 &\; 130
\\
\mathrm{new} &\;
0 &\; 0 &\; 0 &\; 0 &\; 0 &\; 0 &\; 0 &\; 0 &\; 0 &\; 2 &\; 0
\end{array}
}
\]
\caption{Multiplicities of new Lie-Lie identities in degree 6}
\label{LL6mults}
\end{figure}

\begin{theorem} \label{LL6theorem}
Every multilinear identity in degree 6 satisfied by the two Lie brackets \eqref{2lie} 
in the free associative interchange algebra is a consequence of:
\begin{enumerate}
\item[(i)]
anti-commutativity and the Jacobi identity for each Lie bracket;
\item[(ii)]
the following identity which is an alternating sum over all permutations $\sg$ of $\{b,c,d,e\}$ 
where $\epsilon(\sg)$ is the sign:
  \begin{equation}
  \label{newLL6}
  \left\{
  \begin{array}{r@{}l@{}}
  \sum_{\sg} 
  \epsilon(\sg)
  \Big\{ \,
  &
   2  [ [ a b^\sg ]_\hor [ a [ c^\sg [ d^\sg e^\sg ]_\hor ]_\ver ]_\hor ]_\ver
  -2  [ [ a b^\sg ]_\hor [ c^\sg [ a [ d^\sg e^\sg ]_\hor ]_\ver ]_\hor ]_\ver
  \\
  &
  +2  [ [ a b^\sg ]_\hor [ c^\sg [ d^\sg [ a e^\sg ]_\hor ]_\ver ]_\hor ]_\ver
  -   [ [ b^\sg c^\sg ]_\hor [ a [ d^\sg [ a e^\sg ]_\hor ]_\ver ]_\hor ]_\ver
  \\[2pt]
  &
  +   [ [ b^\sg c^\sg ]_\hor [ d^\sg [ a [ a e^\sg ]_\hor ]_\ver ]_\hor ]_\ver
  +   [ [ a b^\sg ]_\hor [ [ a c^\sg ]_\hor [ d^\sg e^\sg ]_\ver ]_\hor ]_\ver
  \\[2pt]
  &
  +2  [ [ a b^\sg ]_\hor [ [ c^\sg d^\sg ]_\hor [ a e^\sg ]_\ver ]_\hor ]_\ver 
  +   [ [ b^\sg c^\sg ]_\hor [ [ a d^\sg ]_\hor [ a e^\sg ]_\ver ]_\hor ]_\ver 
  \\[2pt]
  &
  +   [ [ a [ a b^\sg ]_\hor ]_\hor [ c^\sg [ d^\sg e^\sg ]_\ver ]_\hor ]_\ver
  -   [ [ a [ b^\sg c^\sg ]_\hor ]_\hor [ d^\sg [ a e^\sg ]_\ver ]_\hor ]_\ver
  \\
  &
  +   [ [ b^\sg [ a c^\sg ]_\hor ]_\hor [ a [ d^\sg e^\sg ]_\ver ]_\hor ]_\ver
  \, \Big\}
  \equiv 0;
  \end{array}
  \right.
  \end{equation}
\item[(ii)]
the identity obtained from equation \eqref{newLL6} by interchanging $\hor$ and $\ver$.
\end{enumerate}
\end{theorem}

\begin{lemma} \label{lemmaLL7}
For every partition $\lambda$ of 7, the multiplicity of the simple $S_7$-module $[\lambda]$ in 
the modules $\mathrm{All}(7)$, $\mathrm{Old}(7)$, $\mathrm{New}(7)$ appears in Figure \ref{LL7mults}.
There are two new irreducible multilinear identities in degree 7, relating the two Lie brackets
in the free associative interchange algebra, which are not consequences of the identities of 
lower degree (Theorem \ref{LL6theorem}).
These new identities occur for the last two partitions, $\lambda = 21^5$ and $\lambda = 1^7$. 
\end{lemma}

\begin{figure}[ht]
\small
\[
\boxed{
\begin{array}{c}
\begin{array}{c|ccccccc}
\lambda &\;\quad 
7 &\;\quad  61 &\;\quad  52 &\;\quad  51^2 &\;\quad  43 &\;\quad  421 &\;\quad  41^3 
\\
\mathrm{all} &\;\quad
616 &\;\quad 3632 &\;\quad 8284 &\;\quad 8844 &\;\quad 8128 &\;\quad 20140 &\;\quad 11416 
\\
\mathrm{old} &\;\quad
616 &\;\quad 3632 &\;\quad 8284 &\;\quad 8844 &\;\quad 8128 &\;\quad 20140 &\;\quad 11416 
\\
\mathrm{new} &\;\quad
0 &\;\quad 0 &\;\quad 0 &\;\quad 0 &\;\quad 0 &\;\quad 0 &\;\quad 0
\end{array}
\\
\midrule
\begin{array}{cccccccc}
3^21 &\;\quad  32^2 &\;\quad  321^2 &\;\quad  31^4 &\;\quad  2^31 &\;\quad  2^21^3 &\;\quad  21^5 &\;\quad  1^7
\\
11906 &\;\quad 11804 &\;\quad 19550 &\;\quad 8276 &\;\quad 7672 &\;\quad 7602 &\;\quad 3211 &\;\quad 525
\\
11906 &\;\quad 11804 &\;\quad 19550 &\;\quad 8276 &\;\quad 7672 &\;\quad 7602 &\;\quad 3210 &\;\quad 524
\\
0 &\;\quad 0 &\;\quad 0 &\;\quad 0 &\;\quad 0 &\;\quad 0 &\;\quad 1 &\;\quad 1
\end{array}
\end{array}
}
\]
\caption{Multiplicities of new Lie-Lie identities in degree 7}
\label{LL7mults}
\end{figure}

\begin{theorem} \label{LL7theorem}
Every multilinear identity in degree 7 satisfied by the two Lie brackets \eqref{2lie} 
in the free associative interchange algebra is a consequence of:
\begin{enumerate}
\item[(i)]
anti-commutativity and the Jacobi identity for each Lie bracket;
\item[(ii)]
the identity of equation \eqref{newLL6} and its image under the interchange of $\hor$ and $\ver$;
\item[(iii)]
the new nonlinear identity for $\lambda = 21^5$ in Figure \ref{newlielie7} which has 56 terms;
\item[(iv)]
the following new identity for $\lambda = 1^7$ which is an alternating sum over all permutations $\sg$ 
of $\{a,b,c,d,e,f,g\}$ where $\epsilon(\sg)$ is the sign:
\begin{equation}
\label{newLL7lastpartition}
\left\{ \quad
\begin{array}{l}
\sum_{\sg}
\epsilon(\sg)
\Big\{
  [ [ a^\sg b^\sg ]_\hor [ [ c^\sg d^\sg ]_\ver [ e^\sg [ f^\sg g^\sg ]_\hor ]_\ver ]_\hor ]_\ver 
\\[2pt]
\qquad
+ [ [ a^\sg b^\sg ]_\hor [ [ c^\sg d^\sg ]_\ver [ e^\sg [ f^\sg g^\sg ]_\ver ]_\hor ]_\hor ]_\ver 
\\[2pt]
\qquad\qquad
- 
[ [ a^\sg b^\sg ]_\ver [ [ c^\sg d^\sg ]_\hor [ e^\sg [ f^\sg g^\sg ]_\hor ]_\ver ]_\ver ]_\hor 
\\[2pt]
\qquad\qquad\qquad
-  
[ [ a^\sg b^\sg ]_\ver [ [ c^\sg d^\sg ]_\hor [ e^\sg [ f^\sg g^\sg ]_\ver ]_\hor ]_\ver ]_\hor 
\Big\}
\equiv 0.
\end{array}
\right.
\end{equation}
\end{enumerate}
\end{theorem}

\begin{figure}[ht]
\small
\[
\boxed{
\begin{array}{l}
 2 [ [ a b ]_\hor [ a [ [ c d ]_\hor [ e f ]_\ver ]_\ver ]_\hor ]_\ver  
{} +2 [ [ b c ]_\hor [ a [ [ a d ]_\hor [ e f ]_\ver ]_\ver ]_\hor ]_\ver  
{} +2 [ [ b c ]_\hor [ a [ [ d e ]_\hor [ a f ]_\ver ]_\ver ]_\hor ]_\ver  
\\[2pt] 
{} +4 [ [ a b ]_\hor [ [ a c ]_\ver [ d [ e f ]_\hor ]_\ver ]_\hor ]_\ver  
{} -4 [ [ a b ]_\hor [ [ c d ]_\ver [ e [ a f ]_\hor ]_\ver ]_\hor ]_\ver  
{} -2 [ [ b c ]_\hor [ [ a d ]_\ver [ a [ e f ]_\hor ]_\ver ]_\hor ]_\ver  
\\[2pt] 
{} -2 [ [ b c ]_\hor [ [ d e ]_\ver [ a [ a f ]_\hor ]_\ver ]_\hor ]_\ver  
{} +2 [ [ a b ]_\hor [ [ a c ]_\ver [ d [ e f ]_\ver ]_\hor ]_\hor ]_\ver  
{} +2 [ [ a b ]_\hor [ [ c d ]_\ver [ a [ e f ]_\ver ]_\hor ]_\hor ]_\ver  
\\[2pt] 
{} -2 [ [ a b ]_\hor [ [ c d ]_\ver [ e [ a f ]_\ver ]_\hor ]_\hor ]_\ver  
{}  + [ [ b c ]_\hor [ [ a d ]_\ver [ a [ e f ]_\ver ]_\hor ]_\hor ]_\ver  
{}  + [ [ b c ]_\hor [ [ d e ]_\ver [ a [ a f ]_\ver ]_\hor ]_\hor ]_\ver  
\\[2pt] 
{} +2 [ [ a b ]_\ver [ a [ [ c d ]_\hor [ e f ]_\ver ]_\hor ]_\ver ]_\hor  
{} +2 [ [ b c ]_\ver [ a [ [ a d ]_\hor [ e f ]_\ver ]_\hor ]_\ver ]_\hor  
{} +2 [ [ b c ]_\ver [ a [ [ d e ]_\hor [ a f ]_\ver ]_\hor ]_\ver ]_\hor  
\\[2pt] 
{} -2 [ [ a b ]_\ver [ [ a c ]_\hor [ d [ e f ]_\hor ]_\ver ]_\ver ]_\hor  
{} -2 [ [ a b ]_\ver [ [ c d ]_\hor [ a [ e f ]_\hor ]_\ver ]_\ver ]_\hor  
{} +2 [ [ a b ]_\ver [ [ c d ]_\hor [ e [ a f ]_\hor ]_\ver ]_\ver ]_\hor  
\\[2pt] 
{}  - [ [ b c ]_\ver [ [ a d ]_\hor [ a [ e f ]_\hor ]_\ver ]_\ver ]_\hor  
{}  - [ [ b c ]_\ver [ [ d e ]_\hor [ a [ a f ]_\hor ]_\ver ]_\ver ]_\hor  
{} -4 [ [ a b ]_\ver [ [ a c ]_\hor [ d [ e f ]_\ver ]_\hor ]_\ver ]_\hor  
\\[2pt] 
{} +4 [ [ a b ]_\ver [ [ c d ]_\hor [ e [ a f ]_\ver ]_\hor ]_\ver ]_\hor  
{} +2 [ [ b c ]_\ver [ [ a d ]_\hor [ a [ e f ]_\ver ]_\hor ]_\ver ]_\hor  
{} +2 [ [ b c ]_\ver [ [ d e ]_\hor [ a [ a f ]_\ver ]_\hor ]_\ver ]_\hor  
\\[2pt] 
{}  + [ [ a [ b c ]_\hor ]_\ver [ a [ d [ e f ]_\hor ]_\ver ]_\ver ]_\hor  
{}  + [ [ a [ b c ]_\hor ]_\ver [ d [ a [ e f ]_\hor ]_\ver ]_\ver ]_\hor  
{} -3 [ [ b [ a c ]_\hor ]_\ver [ a [ d [ e f ]_\hor ]_\ver ]_\ver ]_\hor  
\\[2pt] 
{}  - [ [ b [ a c ]_\hor ]_\ver [ d [ a [ e f ]_\hor ]_\ver ]_\ver ]_\hor  
{} -2 [ [ b [ c d ]_\hor ]_\ver [ a [ a [ e f ]_\hor ]_\ver ]_\ver ]_\hor  
{} +3 [ [ b [ c d ]_\hor ]_\ver [ a [ e [ a f ]_\hor ]_\ver ]_\ver ]_\hor  
\\[2pt] 
{}  + [ [ b [ c d ]_\hor ]_\ver [ e [ a [ a f ]_\hor ]_\ver ]_\ver ]_\hor  
{} -4 [ [ a [ a b ]_\hor ]_\ver [ c [ d [ e f ]_\ver ]_\hor ]_\ver ]_\hor  
{} -2 [ [ a [ b c ]_\hor ]_\ver [ a [ d [ e f ]_\ver ]_\hor ]_\ver ]_\hor  
\\[2pt] 
{} +2 [ [ a [ b c ]_\hor ]_\ver [ d [ a [ e f ]_\ver ]_\hor ]_\ver ]_\hor  
{} -4 [ [ a [ b c ]_\hor ]_\ver [ d [ e [ a f ]_\ver ]_\hor ]_\ver ]_\hor  
{} +2 [ [ a [ a b ]_\hor ]_\ver [ [ c d ]_\hor [ e f ]_\ver ]_\ver ]_\hor  
\\[2pt] 
{} +2 [ [ a [ b c ]_\hor ]_\ver [ [ a d ]_\hor [ e f ]_\ver ]_\ver ]_\hor  
{}  + [ [ a [ b c ]_\hor ]_\ver [ [ d e ]_\hor [ a f ]_\ver ]_\ver ]_\hor  
{}  + [ [ b [ a c ]_\hor ]_\ver [ [ d e ]_\hor [ a f ]_\ver ]_\ver ]_\hor  
\\[2pt] 
{}  + [ [ b [ c d ]_\hor ]_\ver [ [ a e ]_\hor [ a f ]_\ver ]_\ver ]_\hor  
{} +4 [ [ a [ a b ]_\ver ]_\hor [ c [ d [ e f ]_\hor ]_\ver ]_\hor ]_\ver  
{} +2 [ [ a [ b c ]_\ver ]_\hor [ a [ d [ e f ]_\hor ]_\ver ]_\hor ]_\ver  
\\[2pt] 
{} -2 [ [ a [ b c ]_\ver ]_\hor [ d [ a [ e f ]_\hor ]_\ver ]_\hor ]_\ver  
{} +4 [ [ a [ b c ]_\ver ]_\hor [ d [ e [ a f ]_\hor ]_\ver ]_\hor ]_\ver  
{}  - [ [ a [ b c ]_\ver ]_\hor [ a [ d [ e f ]_\ver ]_\hor ]_\hor ]_\ver  
\\[2pt] 
{}  - [ [ a [ b c ]_\ver ]_\hor [ d [ a [ e f ]_\ver ]_\hor ]_\hor ]_\ver  
{} +3 [ [ b [ a c ]_\ver ]_\hor [ a [ d [ e f ]_\ver ]_\hor ]_\hor ]_\ver  
{}  + [ [ b [ a c ]_\ver ]_\hor [ d [ a [ e f ]_\ver ]_\hor ]_\hor ]_\ver  
\\[2pt] 
{} +2 [ [ b [ c d ]_\ver ]_\hor [ a [ a [ e f ]_\ver ]_\hor ]_\hor ]_\ver  
{} -3 [ [ b [ c d ]_\ver ]_\hor [ a [ e [ a f ]_\ver ]_\hor ]_\hor ]_\ver  
{}  - [ [ b [ c d ]_\ver ]_\hor [ e [ a [ a f ]_\ver ]_\hor ]_\hor ]_\ver  
\\[2pt] 
{} +2 [ [ a [ a b ]_\ver ]_\hor [ [ c d ]_\hor [ e f ]_\ver ]_\hor ]_\ver  
{}  + [ [ a [ b c ]_\ver ]_\hor [ [ a d ]_\hor [ e f ]_\ver ]_\hor ]_\ver  
{} +2 [ [ a [ b c ]_\ver ]_\hor [ [ d e ]_\hor [ a f ]_\ver ]_\hor ]_\ver  
\\[2pt] 
{}  + [ [ b [ a c ]_\ver ]_\hor [ [ a d ]_\hor [ e f ]_\ver ]_\hor ]_\ver  
{}  - [ [ b [ c d ]_\ver ]_\hor [ [ a e ]_\hor [ a f ]_\ver ]_\hor ]_\ver  
\end{array}
}
\]
\caption{New nonlinear Lie-Lie identity in degree 7 for $\lambda = 21^5$}
\label{newlielie7}
\end{figure}


\section{Two Jordan products} \label{sectionjordanjordan}

In this section we study the multilinear identities relating the two Jordan products:
  \begin{equation} \label{2jordan}
  \jorl a, b \jorr_\hor = a \hor b + b \hor a,
  \qquad\qquad
  \jorl a, b \jorr_\ver = a \ver b + b \ver a.
  \end{equation}

\begin{proposition} \label{propositionJJ4}
Every multilinear identity in degree 4 satisfied by the Jordan products \eqref{2jordan} in the free 
associative interchange algebra is a consequence of the following identities:
\begin{enumerate}
\item[(i)]
commutativity and the (linearized) Jordan identity for each Jordan product;
\item[(ii)]
the linearization of the following new identity relating the two Jordan products:
  \begin{equation}
  \label{newJJ4}
  \jorl \jorl a, a \jorr_\hor, \jorl a, a \jorr_\hor \jorr_\ver 
  - 
  \jorl \jorl a, a \jorr_\ver, \jorl a, a \jorr_\ver \jorr_\hor
  \equiv 0.
  \end{equation}
\end{enumerate}
\end{proposition}

\begin{proof}
Computations using the representation theory of $S_4$ show that the matrices $\mathsf{all}(\lambda)$ and
$\mathsf{old}(\lambda)$ are equal for all partitions except the first, $\lambda =  4$.
In this case we have
  \[
  \mathsf{all}(\lambda) 
  = 
  \left[ 
  \begin{array}{rrrrrrrrrrrrrr}
  1 &\; . &\; . &\; . &\; . &\; . &\; . &\; . &\; -1 &\; . &\; . &\; . &\; . &\; . \\[-1pt]
  . &\; . &\; . &\; . &\; . &\; . &\; . &\; 1 &\; . &\; . &\; . &\; . &\; . &\; -1 \\[-1pt]
  . &\; . &\; . &\; . &\; . &\; . &\; . &\; . &\; . &\; 1 &\; . &\; . &\; -1 &\; 0
  \end{array} 
  \right],
  \]
and $\mathsf{old}(\lambda)$ consists of the first two rows of $\mathsf{all}(\lambda)$.
Hence the third row of $\mathsf{all}(\lambda)$ represents a new multilinear identity $I$, which is 
the difference of the symmetric sums over all permutations of the variables in JJ types 10 and 13
(the positions of the nonzero entries in row 3).
These are the JJ types in the nonlinear identity \eqref{newJJ4}: in characteristic 0 or $p > 4$ this 
nonlinear identity is equivalent to the multilinear identity $I$.

Identity \eqref{newJJ4} is a Jordan analogue of the interchange identity \eqref{interchange}:
it states the equality of two monomials with the same association types as the interchange identity, 
but using Jordan products instead.
However, \eqref{newJJ4} involves only one variable whereas \eqref{interchange} is multilinear
in four variables; 
we can make \eqref{newJJ4} multilinear by replacing each monomial by the symmetric sum over all 
permutations of $\{a,b,c,d\}$.
\end{proof}

\begin{proposition} \label{JJ5proposition}
Every multilinear identity in degree 5 satisfied by the Jordan products \eqref{2jordan} in the free 
associative interchange algebra is a consequence of the following identities:
\begin{enumerate}
\item[(i)]
commutativity and the (linearized) Jordan identity for each Jordan product;
\item[(ii)]
the linearization of identity \eqref{newJJ4} of Proposition \ref{propositionJJ4};
\item[(iii)]
the linearizations of these two new identities relating the Jordan products:
  \begin{align}
  \label{newJJ51}
  &\;
  \jorl \jorl a, b \jorr_\hor, \jorl a, \jorl a, a \jorr_\hor \jorr_\hor \jorr_\ver 
  - 
  \jorl \jorl a, a \jorr_\ver, \jorl b, \jorl a, a \jorr_\hor \jorr_\ver \jorr_\hor
  \equiv 0,
  \\
  \label{newJJ52}
  &\;
  \jorl \jorl a, a \jorr_\hor, \jorl b, \jorl a, a \jorr_\ver \jorr_\hor \jorr_\ver 
  - 
  \jorl \jorl a, b \jorr_\ver, \jorl a, \jorl a, a \jorr_\ver \jorr_\ver \jorr_\hor
  \equiv 0.
  \end{align}
\end{enumerate}
\end{proposition}

\begin{proof}
In degree 5, our computations show that $\mathsf{all}(\lambda) = \mathsf{old}(\lambda)$ for all 
partitions except the second, $\lambda = 41$.
In this case, $\mathrm{rank}(\mathsf{old}(\lambda)) = 102$ and $\mathrm{rank}(\mathsf{all}(\lambda)) = 104$, 
so we expect two new independent identities which are equivalent (in characteristic 0 or $p > 5$)   
to nonlinear identities which are linear combinations of JJ monomials whose underlying variables 
are permutations of $a^4b$.
Rows 73 and 83 of $\mathsf{all}(\lambda)$ represent these new identities: these are the rows whose 
leading 1s occur in columns for which $\mathsf{old}(\lambda)$ has no leading 1.
Each of these rows of $\mathsf{all}(\lambda)$ has only two nonzero entries:
  \begin{equation}
  \label{4nonzero}
  \left\{ \;
  \begin{array}{l}
  \text{row 73}\colon \qquad \mathsf{all}(\lambda)_{73,119} = 1, \qquad \mathsf{all}(\lambda)_{73,155} = 2, 
  \\
  \text{row 83}\colon \qquad \mathsf{all}(\lambda)_{83,135} = 1, \qquad \mathsf{all}(\lambda)_{83,171} = \tfrac12.
  \end{array}
  \right.
  \end{equation}
In degree 5, there are 44 JJ types, and $\dim [\lambda] = 4$, so each row of $\mathsf{all}(\lambda)$ 
has 176 entries in 44 segments of length 4.
For $t = 1, \dots, 44$ columns $4t{-}3 \le j \le 4t$ form the segment corresponding to the representation 
matrix units \eqref{Sndecomp} with JJ type $t$.
The entries \eqref{4nonzero} occur in position 3 of segments 30, 39, 34, 43 respectively.

By the representation theory of $S_5$, the matrix unit in 
position $(1,3)$ of the $4 \times 4$ matrix for partition $\lambda = 41$ corresponds under the isomorphism
$R$ of equation \eqref{Sndecomp} to the following element of $\mathbb{Q} S_5$, 
in which all permutations $\sg$ are written in the form $a^\sg b^\sg c^\sg d^\sg e^\sg$:
  \begin{equation}
  \label{Snelement}
  \tfrac16 
  \big( \sum_{\sg \in S_4} a^\sg b^\sg c^\sg d^\sg e \big)
  \big( abcde - ebcda \big) \;
  abecd.
  \end{equation}
This is the result of applying $R^{-1}$ to the $4 \times 4$ matrix unit 
$E_{13}$ in the simple two-sided ideal for partition $\lambda = 41$.
Each entry \eqref{4nonzero} represents a scalar multiple of the corresponding JJ type applied to \eqref{Snelement}. 
Hence the two new multilinear identities each have 96 terms; commutativity of the Jordan products reduces this to 24.

We can find simpler nonlinear identities as follows.
We apply each of the JJ types 30, 39, 34, 43 to the 5 permutations of $a^4b$ and use Jordan commutativity
to get three nonlinear monomials for each JJ type.
We expand these 12 JJ monomials into the free associative interchange algebra and identify the 12 AA types 
which occur in the terms of the expansions.
We determine the union of the equivalence classes containing these 12 AA types.
(The equivalence relation on AA types is generated by the consequences of the interchange identity
\cite{BM}: it is the reflexive transitive closure of the relation defined by $t \sim t'$ if and only if 
there is a consequence of \eqref{interchange} whose two terms have AA types $t$ and $t'$.)
Only 8 of the 12 consequences of the interchange identity in degree 5 involve these 12 AA types.
The rest of the calculation is similar to the multilinear case.
\end{proof}

\begin{lemma} \label{lemmaJJ6}
For every partition $\lambda$ of 6, the multiplicity of the simple $S_6$-module $[\lambda]$ in 
the modules $\mathrm{All}(6)$, $\mathrm{Old}(6)$, $\mathrm{New}(6)$ appears in Figure \ref{JJ6mults}.
Summing the multiplicities gives 14 new irreducible multilinear identities in degree 6, 
relating the two Jordan products \eqref{2jordan} in the free associative interchange algebra, 
which are not consequences of the identities of lower degree (Proposition \ref{JJ5proposition}).
These new identities occur for the first 5 partitions.
\end{lemma}

\begin{figure}[ht]
\small
\[
\boxed{
\begin{array}{c|ccccccccccc}
\lambda &\; 
6 &\;  51 &\;  42 &\;  41^2 &\;  3^2  &\; 321 &\;  31^3 &\;  2^3 &\;  2^21^2 &\;  21^4 &\;  1^6
\\
\mathrm{all} &\;
75 &\; 541 &\; 1031 &\; 1286 &\; 615 &\; 2088 &\; 1446 &\; 653 &\; 1302 &\; 788 &\; 164
\\
\mathrm{old} &\;
73 &\; 536 &\; 1027 &\; 1285 &\; 613 &\; 2088 &\; 1446 &\; 653 &\; 1302 &\; 788 &\; 164
\\
\mathrm{new} &\;
2 &\; 5 &\; 4 &\; 1 &\; 2 &\; 0 &\; 0 &\; 0 &\; 0 &\; 0 &\; 0
\end{array}
}
\]
\caption{Multiplicities of new Jordan-Jordan identities in degree 6}
\label{JJ6mults}
\end{figure}

\begin{figure}[ht]
\small
\[
\boxed{
\begin{array}{ll} 
6 &
   \jorl \jorl a a \jorr_\hor \jorl a \jorl a \jorl a a \jorr_\ver \jorr_\hor \jorr_\ver \jorr_\ver  
 - \jorl \jorl a \jorl a a \jorr_\hor \jorr_\ver \jorl a \jorl a a \jorr_\ver \jorr_\hor \jorr_\ver  
\\
\midrule
&
   \jorl \jorl a a \jorr_\ver \jorl a \jorl a \jorl a a \jorr_\hor \jorr_\ver \jorr_\hor \jorr_\hor  
 - \jorl \jorl a \jorl a a \jorr_\hor \jorr_\ver \jorl a \jorl a a \jorr_\ver \jorr_\hor \jorr_\hor  
\\
\midrule
51 &
   \jorl \jorl a a \jorr_\hor \jorl b \jorl a \jorl a a \jorr_\ver \jorr_\hor \jorr_\ver \jorr_\ver  
 - \jorl \jorl b \jorl a a \jorr_\hor \jorr_\ver \jorl a \jorl a a \jorr_\ver \jorr_\hor \jorr_\ver  
\\
\midrule
&
   \jorl \jorl a a \jorr_\ver \jorl b \jorl a \jorl a a \jorr_\hor \jorr_\ver \jorr_\hor \jorr_\hor  
 - \jorl \jorl a \jorl a a \jorr_\hor \jorr_\ver \jorl b \jorl a a \jorr_\ver \jorr_\hor \jorr_\hor  
\\
\midrule
&
   \jorl \jorl a \jorl a a \jorr_\hor \jorr_\hor \jorl b \jorl a a \jorr_\ver \jorr_\hor \jorr_\ver  
 - \jorl \jorl b \jorl a a \jorr_\hor \jorr_\ver \jorl a \jorl a a \jorr_\ver \jorr_\ver \jorr_\hor  
\\
\midrule
&
   \jorl \jorl a a \jorr_\ver \jorl a \jorl b \jorl a a \jorr_\hor \jorr_\ver \jorr_\ver \jorr_\hor  
 - \jorl \jorl a a \jorr_\ver \jorl b \jorl a \jorl a a \jorr_\hor \jorr_\ver \jorr_\ver \jorr_\hor  
 + \jorl \jorl a \jorl a a \jorr_\hor \jorr_\hor \jorl a \jorl a b \jorr_\ver \jorr_\hor \jorr_\ver  
\\ 
&
-2 \jorl \jorl a \jorl a a \jorr_\hor \jorr_\ver \jorl a \jorl a b \jorr_\ver \jorr_\ver \jorr_\hor  
 + \jorl \jorl a \jorl a a \jorr_\hor \jorr_\ver \jorl b \jorl a a \jorr_\ver \jorr_\ver \jorr_\hor  
\\
\midrule
&
 2 \jorl \jorl a a \jorr_\ver \jorl a \jorl b \jorl a a \jorr_\ver \jorr_\hor \jorr_\ver \jorr_\hor  
+2 \jorl \jorl a a \jorr_\ver \jorl b \jorl a \jorl a a \jorr_\ver \jorr_\hor \jorr_\ver \jorr_\hor  
-4 \jorl \jorl a b \jorr_\ver \jorl a \jorl a \jorl a a \jorr_\ver \jorr_\hor \jorr_\ver \jorr_\hor  
\\ 
&
-2 \jorl \jorl a \jorl a b \jorr_\hor \jorr_\hor \jorl a \jorl a a \jorr_\ver \jorr_\hor \jorr_\ver  
+3 \jorl \jorl b \jorl a a \jorr_\hor \jorr_\hor \jorl a \jorl a a \jorr_\ver \jorr_\hor \jorr_\ver  
-2 \jorl \jorl a \jorl a a \jorr_\hor \jorr_\ver \jorl b \jorl a a \jorr_\ver \jorr_\ver \jorr_\hor  
\\ 
&
 + \jorl \jorl b \jorl a a \jorr_\hor \jorr_\ver \jorl a \jorl a a \jorr_\ver \jorr_\ver \jorr_\hor  
\\
\midrule
42
&
 2 \jorl \jorl a a \jorr_\hor \jorl \jorl a b \jorr_\hor \jorl a b \jorr_\hor \jorr_\ver \jorr_\ver  
 - \jorl \jorl a \jorl a a \jorr_\ver \jorr_\ver \jorl a \jorl b b \jorr_\ver \jorr_\ver \jorr_\hor  
-2 \jorl \jorl a \jorl a b \jorr_\ver \jorr_\ver \jorl b \jorl a a \jorr_\ver \jorr_\ver \jorr_\hor  
\\ 
&
 + \jorl \jorl b \jorl a a \jorr_\ver \jorr_\ver \jorl b \jorl a a \jorr_\ver \jorr_\ver \jorr_\hor  
\\
\midrule
&
   \jorl \jorl b b \jorr_\ver \jorl \jorl a a \jorr_\hor \jorl a a \jorr_\ver \jorr_\ver \jorr_\hor  
+2 \jorl \jorl a \jorl a b \jorr_\hor \jorr_\hor \jorl b \jorl a a \jorr_\ver \jorr_\hor \jorr_\ver  
 - \jorl \jorl b \jorl a a \jorr_\hor \jorr_\hor \jorl b \jorl a a \jorr_\ver \jorr_\hor \jorr_\ver  
\\ 
&
-2 \jorl \jorl b \jorl a b \jorr_\hor \jorr_\ver \jorl a \jorl a a \jorr_\ver \jorr_\ver \jorr_\hor  
\\
\midrule
&
 2 \jorl \jorl a a \jorr_\ver \jorl \jorl a b \jorr_\ver \jorl a b \jorr_\ver \jorr_\hor \jorr_\hor  
 - \jorl \jorl a \jorl a a \jorr_\hor \jorr_\hor \jorl a \jorl b b \jorr_\hor \jorr_\hor \jorr_\ver  
-2 \jorl \jorl a \jorl a b \jorr_\hor \jorr_\hor \jorl b \jorl a a \jorr_\hor \jorr_\hor \jorr_\ver  
\\ 
&
 + \jorl \jorl b \jorl a a \jorr_\hor \jorr_\hor \jorl b \jorl a a \jorr_\hor \jorr_\hor \jorr_\ver  
\\
\midrule
&
 2 \jorl \jorl a a \jorr_\ver \jorl a \jorl b \jorl a b \jorr_\hor \jorr_\ver \jorr_\ver \jorr_\hor  
-2 \jorl \jorl a a \jorr_\ver \jorl b \jorl a \jorl a b \jorr_\hor \jorr_\ver \jorr_\ver \jorr_\hor  
+2 \jorl \jorl a a \jorr_\ver \jorl b \jorl b \jorl a a \jorr_\ver \jorr_\hor \jorr_\ver \jorr_\hor  
\\ 
&
-2 \jorl \jorl b b \jorr_\ver \jorl a \jorl a \jorl a a \jorr_\ver \jorr_\hor \jorr_\ver \jorr_\hor  
-2 \jorl \jorl a a \jorr_\ver \jorl \jorl a b \jorr_\hor \jorl a b \jorr_\ver \jorr_\ver \jorr_\hor  
 + \jorl \jorl b b \jorr_\ver \jorl \jorl a a \jorr_\hor \jorl a a \jorr_\ver \jorr_\ver \jorr_\hor  
\\ 
&
 - \jorl \jorl a \jorl a a \jorr_\hor \jorr_\ver \jorl a \jorl b b \jorr_\ver \jorr_\ver \jorr_\hor  
+2 \jorl \jorl a \jorl a a \jorr_\hor \jorr_\ver \jorl b \jorl a b \jorr_\ver \jorr_\ver \jorr_\hor  
\\
\midrule
41^2
&
\sum_{\sg \in S_2} 
\epsilon(\sigma)
\big[
   \jorl \jorl a a \jorr_\hor \jorl b^\sigma \jorl c^\sigma \jorl a a \jorr_\ver \jorr_\hor \jorr_\hor \jorr_\ver  
 + \jorl \jorl a a \jorr_\hor \jorl b^\sigma \jorl a \jorl a c^\sigma \jorr_\ver \jorr_\hor \jorr_\hor \jorr_\ver  
\\ 
&
 + \jorl \jorl a a \jorr_\ver \jorl b^\sigma \jorl c^\sigma \jorl a a \jorr_\hor \jorr_\ver \jorr_\ver \jorr_\hor  
 + \jorl \jorl a a \jorr_\ver \jorl b^\sigma \jorl a \jorl a c^\sigma \jorr_\hor \jorr_\ver \jorr_\ver \jorr_\hor  
\\ 
&
 - \jorl \jorl a a \jorr_\ver \jorl a \jorl b^\sigma \jorl a c^\sigma \jorr_\hor \jorr_\ver \jorr_\ver \jorr_\hor  
 - \jorl \jorl a a \jorr_\ver \jorl a \jorl b^\sigma \jorl a c^\sigma \jorr_\ver \jorr_\hor \jorr_\ver \jorr_\hor  
\\ 
&
-2 \jorl \jorl b^\sigma \jorl a c^\sigma \jorr_\hor \jorr_\hor \jorl a \jorl a a \jorr_\ver \jorr_\hor \jorr_\ver  
+2 \jorl \jorl a \jorl a b^\sigma \jorr_\hor \jorr_\hor \jorl a \jorl a c^\sigma \jorr_\ver \jorr_\hor \jorr_\ver  
\\ 
&
-2 \jorl \jorl a \jorl a b^\sigma \jorr_\hor \jorr_\ver \jorl a \jorl a c^\sigma \jorr_\ver \jorr_\ver \jorr_\hor  
-2 \jorl \jorl a \jorl a a \jorr_\hor \jorr_\ver \jorl b^\sigma \jorl a c^\sigma \jorr_\ver \jorr_\ver \jorr_\hor  
\big]
\\
\midrule
3^2
&
 2 \jorl \jorl a b \jorr_\ver \jorl \jorl b b \jorr_\hor \jorl a a \jorr_\ver \jorr_\ver \jorr_\hor  
-2 \jorl \jorl a \jorl b b \jorr_\hor \jorr_\hor \jorl b \jorl a a \jorr_\ver \jorr_\hor \jorr_\ver  
 - \jorl \jorl b \jorl b b \jorr_\hor \jorr_\hor \jorl a \jorl a a \jorr_\ver \jorr_\hor \jorr_\ver  
\\ 
&
 + \jorl \jorl b \jorl b b \jorr_\hor \jorr_\ver \jorl a \jorl a a \jorr_\ver \jorr_\ver \jorr_\hor  
\\
\midrule
&
   \jorl \jorl b b \jorr_\ver \jorl a \jorl b \jorl a a \jorr_\hor \jorr_\ver \jorr_\ver \jorr_\hor  
 - \jorl \jorl b b \jorr_\ver \jorl b \jorl a \jorl a a \jorr_\hor \jorr_\ver \jorr_\ver \jorr_\hor  
-2 \jorl \jorl a b \jorr_\ver \jorl b \jorl b \jorl a a \jorr_\ver \jorr_\hor \jorr_\ver \jorr_\hor  
\\ 
&
 + \jorl \jorl b b \jorr_\ver \jorl a \jorl b \jorl a a \jorr_\ver \jorr_\hor \jorr_\ver \jorr_\hor  
 + \jorl \jorl b b \jorr_\ver \jorl b \jorl a \jorl a a \jorr_\ver \jorr_\hor \jorr_\ver \jorr_\hor  
 - \jorl \jorl a b \jorr_\ver \jorl \jorl b b \jorr_\hor \jorl a a \jorr_\ver \jorr_\ver \jorr_\hor  
\\ 
&
 - \jorl \jorl b b \jorr_\ver \jorl \jorl a a \jorr_\hor \jorl a b \jorr_\ver \jorr_\ver \jorr_\hor  
 - \jorl \jorl b b \jorr_\ver \jorl \jorl a b \jorr_\hor \jorl a a \jorr_\ver \jorr_\ver \jorr_\hor  
 + \jorl \jorl a \jorl b b \jorr_\hor \jorr_\ver \jorl b \jorl a a \jorr_\ver \jorr_\ver \jorr_\hor  
\\ 
&
+2 \jorl \jorl b \jorl a b \jorr_\hor \jorr_\ver \jorl b \jorl a a \jorr_\ver \jorr_\ver \jorr_\hor  
\end{array}
}
\]
\caption{14 new nonlinear Jordan-Jordan identities in degree 6}
\label{JJ6table}
\end{figure}

\begin{theorem} \label{JJ6theorem}
Every multilinear identity in degree 6 satisfied by the Jordan products \eqref{2jordan} in the free 
associative interchange algebra is a consequence of the following identities:
\begin{enumerate}
\item[(i)]
commutativity and the (linearized) Jordan identity for each Jordan product;
\item[(ii)]
the linearization of identity \eqref{newJJ4} of Proposition \ref{propositionJJ4};
\item[(iii)]
the linearization of identities \eqref{newJJ51} and \eqref{newJJ52} of Lemma \ref{JJ5proposition};
\item[(iv)]
the linearizations of the identities of Figure \ref{JJ6table} which are compact nonlinear forms of 
the 14 new irreducible identities identified by Lemma \ref{lemmaJJ6}.
\end{enumerate}
\end{theorem}

\begin{lemma} \label{JJdegree7}
For every partition $\lambda$ of 7, the multiplicity of the simple $S_7$-module $[\lambda]$ in 
the modules $\mathrm{All}(7)$, $\mathrm{Old}(7)$, $\mathrm{New}(7)$ appears in Figure \ref{JJ7mults}.
Summing the multiplicities gives 94 new irreducible multilinear identities in degree 7,
relating the two Jordan products in the free associative interchange algebra, which are not 
consequences of the identities of lower degree (Theorem \ref{JJ6theorem}).
These new identities occur for partitions 1--6 and 8.
\end{lemma}

\begin{figure}[ht]
\small
\[
\boxed{
\begin{array}{c}
\begin{array}{c|ccccccc}
\lambda &\;\quad 
7 &\;\quad  61 &\;\quad  52 &\;\quad  51^2 &\;\quad  43 &\;\quad  421 &\;\quad  41^3 
\\
\mathrm{all} &\;\quad
349 &\;\quad 2618 &\;\quad 6510 &\;\quad 7444 &\;\quad 6740 &\;\quad 17752 &\;\quad 10842 
\\
\mathrm{old} &\;\quad
347 &\;\quad 2609 &\;\quad 6493 &\;\quad 7424 &\;\quad 6725 &\;\quad 17732 &\;\quad 10842 
\\
\mathrm{new} &\;\quad
2 &\;\quad 9 &\;\quad 17 &\;\quad 20 &\;\quad 15 &\;\quad 20 &\;\quad 0 
\end{array}
\\
\midrule
\begin{array}{cccccccc}
3^21 &\;\quad  32^2 &\;\quad  321^2 &\;\quad  31^4 &\;\quad  2^31 &\;\quad  2^21^3 &\;\quad  21^5 &\;\quad  1^7
\\
10839 &\;\quad 10922 &\;\quad 19088 &\;\quad 8692 &\;\quad 7666 &\;\quad 8116 &\;\quad 3632 &\;\quad 616
\\
10828 &\;\quad 10922 &\;\quad 19088 &\;\quad 8692 &\;\quad 7666 &\;\quad 8116 &\;\quad 3632 &\;\quad 616
\\
11 &\;\quad 0 &\;\quad 0 &\;\quad 0 &\;\quad 0 &\;\quad 0 &\;\quad 0 &\;\quad 0
\end{array}
\end{array}
}
\]
\caption{Multiplicities of new Jordan-Jordan identities in degree 7}
\label{JJ7mults}
\end{figure}

\begin{theorem} \label{JJ7theorem}
Every multilinear identity in degree 7 satisfied by the two Jordan products \eqref{2jordan} 
in the free associative interchange algebra is a consequence of:
\begin{enumerate}
\item[(i)]
commutativity and the (linearized) Jordan identity for each Jordan product;
\item[(ii)]
the linearization of identity \eqref{newJJ4} of Proposition \ref{propositionJJ4};
\item[(iii)]
the linearization of identities \eqref{newJJ51} and \eqref{newJJ52} of Proposition \ref{JJ5proposition};
\item[(iv)]
the linearizations of the identities of Figure \ref{JJ6table} from Theorem \ref{JJ6theorem};
\item[(v)]
the linearizations of the identities of Figure \ref{JJnew7firstidentities} which are nonlinear 
forms of 7 of the 94 new irreducible identities established by Lemma \ref{JJdegree7} (we display
only the simplest identity for each partition 1--6 and 8);
\item[(vi)]
the linearizations of the remaining 89 nonlinear identities which are available in an ancillary file 
for the \emph{\texttt{arXiv}} version of this paper; see \cite{BMarXiv} for details.
\end{enumerate}
\end{theorem}

\begin{proof}
From Lemma \ref{JJdegree7} we see that only seven partitions of 7 provide new JJ identities.
For each of these partitions $\lambda$, Figure \ref{JJnew7info} contains:
  \begin{itemize}
  \item
  The number of new independent identities for partition $\lambda$ (from Figure \ref{JJ7mults}).
  \item
  The number of nonzero entries in the row of the representation matrix corresponding to each new identity.
  \end{itemize}
In order to find explicit new nonlinear identities in degree 7, we use the computational methods 
introduced for one Lie bracket and one Jordan product: see the proofs of Theorem \ref{LJ7theorem} 
and Proposition \ref{JJ5proposition}.
The remaining details of this proof require further explanation of our computational methods and 
are presented in section \ref{sectionJJ7proof}.
\end{proof}

  \begin{figure}[ht]
  \small
  \[
  \boxed{
  \begin{array}{ll}
  7 & 
     \jorl \liel a a \lier \jorl a \liel \liel a a \lier \jorl a a \jorr \lier \jorr \jorr 
  +  \jorl \liel a \liel a a \lier \lier \jorl a \liel a \jorl a a \jorr \lier \jorr \jorr 
  -  \jorl \jorl a \liel a a \lier \jorr \liel \liel a a \lier \jorl a a \jorr \lier \jorr 
  -  \jorl \liel a \jorl a a \jorr \lier \jorl a \liel a \liel a a \lier \lier \jorr \jorr 
  \\ \midrule
  61 &
     \jorl \liel a a \lier \jorl b \liel \liel a a \lier \jorl a a \jorr \lier \jorr \jorr 
  +  \jorl \liel a \liel a a \lier \lier \jorl b \liel a \jorl a a \jorr \lier \jorr \jorr 
  -  \jorl \jorl b \liel a a \lier \jorr \liel \liel a a \lier \jorl a a \jorr \lier \jorr 
  -  \jorl \liel a \jorl a a \jorr \lier \jorl b \liel a \liel a a \lier \lier \jorr \jorr 
  \\ \midrule
  52 &
     \liel \jorl a a \jorr \liel b \jorl b \liel a \liel a a \lier \lier \jorr \lier \lier 
  -  \liel \jorl a b \jorr \liel b \liel \jorl a a \jorr \jorl a a \jorr \lier \lier \lier 
  -2 \liel \jorl a \liel a a \lier \jorr \liel b \jorl b \liel a a \lier \jorr \lier \lier 
  +2 \liel \jorl b \liel a a \lier \jorr \liel b \jorl a \liel a a \lier \jorr \lier \lier 
  \\
  &
  -  \liel \liel b \jorl a a \jorr \lier \jorl b \liel a \liel a a \lier \lier \jorr \lier 
  +  \liel \liel b \jorl a b \jorr \lier \liel \jorl a a \jorr \jorl a a \jorr \lier \lier 
  \\ \midrule
  51^2 &
     \liel \jorl b \jorl a a \jorr \jorr \jorl a \jorl c \liel a a \lier \jorr \jorr \lier 
  +  \liel \jorl b \jorl a a \jorr \jorr \jorl c \jorl a \liel a a \lier \jorr \jorr \lier 
  -  \liel \jorl c \jorl a a \jorr \jorr \jorl a \jorl b \liel a a \lier \jorr \jorr \lier 
  -  \liel \jorl c \jorl a a \jorr \jorr \jorl b \jorl a \liel a a \lier \jorr \jorr \lier 
  \\
  &
  -  \liel \jorl b \jorl a a \jorr \jorr \jorl \liel a a \lier \jorl a c \jorr \jorr \lier 
  +  \liel \jorl c \jorl a a \jorr \jorr \jorl \liel a a \lier \jorl a b \jorr \jorr \lier 
  \\ \midrule
  43 &
     \jorl \liel a \liel b b \lier \lier \liel \liel a b \lier \jorl a a \jorr \lier \jorr 
  +2 \jorl \liel b \liel a b \lier \lier \liel \liel a b \lier \jorl a a \jorr \lier \jorr 
  -2 \liel \jorl b \jorl a a \jorr \jorr \jorl \liel a b \lier \liel a b \lier \jorr \lier 
  -  \liel \jorl a \jorl a a \jorr \jorr \liel \jorl a b \jorr \jorl b b \jorr \lier \lier 
  \\ \midrule
  421 &
     \liel \jorl a \jorl b b \jorr \jorr \jorl a \jorl c \liel a a \lier \jorr \jorr \lier 
  +  \liel \jorl a \jorl b b \jorr \jorr \jorl c \jorl a \liel a a \lier \jorr \jorr \lier 
  -2 \liel \jorl b \jorl a b \jorr \jorr \jorl a \jorl c \liel a a \lier \jorr \jorr \lier 
  -2 \liel \jorl b \jorl a b \jorr \jorr \jorl c \jorl a \liel a a \lier \jorr \jorr \lier 
  \\
  &
  +2 \liel \jorl c \jorl a a \jorr \jorr \jorl b \jorl b \liel a a \lier \jorr \jorr \lier 
  -  \liel \jorl a \jorl b b \jorr \jorr \jorl \liel a a \lier \jorl a c \jorr \jorr \lier 
  +2 \liel \jorl b \jorl a b \jorr \jorr \jorl \liel a a \lier \jorl a c \jorr \jorr \lier 
  -  \liel \jorl c \jorl a a \jorr \jorr \jorl \liel a a \lier \jorl b b \jorr \jorr \lier 
  \\ \midrule
  3^21 &
     \liel \jorl a \jorl b b \jorr \jorr \jorl b \jorl c \liel a a \lier \jorr \jorr \lier 
  +  \liel \jorl a \jorl b b \jorr \jorr \jorl c \jorl b \liel a a \lier \jorr \jorr \lier 
  -2 \liel \jorl a \jorl b c \jorr \jorr \jorl b \jorl b \liel a a \lier \jorr \jorr \lier 
  -2 \liel \jorl b \jorl a b \jorr \jorr \jorl b \jorl c \liel a a \lier \jorr \jorr \lier 
  \\
  &
  -2 \liel \jorl b \jorl a b \jorr \jorr \jorl c \jorl b \liel a a \lier \jorr \jorr \lier 
  +2 \liel \jorl b \jorl a c \jorr \jorr \jorl b \jorl b \liel a a \lier \jorr \jorr \lier 
  +2 \liel \jorl c \jorl a b \jorr \jorr \jorl b \jorl b \liel a a \lier \jorr \jorr \lier 
  -  \liel \jorl a \jorl b b \jorr \jorr \jorl \liel a a \lier \jorl b c \jorr \jorr \lier 
  \\
  &
  +  \liel \jorl a \jorl b c \jorr \jorr \jorl \liel a a \lier \jorl b b \jorr \jorr \lier 
  +2 \liel \jorl b \jorl a b \jorr \jorr \jorl \liel a a \lier \jorl b c \jorr \jorr \lier 
  -  \liel \jorl b \jorl a c \jorr \jorr \jorl \liel a a \lier \jorl b b \jorr \jorr \lier 
  -  \liel \jorl c \jorl a b \jorr \jorr \jorl \liel a a \lier \jorl b b \jorr \jorr \lier 
  \end{array}
  }
  \]
  \caption{The simplest new Jordan-Jordan identities in degree 7}
  \label{JJnew7firstidentities}
  \end{figure}

  \begin{figure}[ht]
  \small
  \[
  \boxed{
  \begin{array}{lrl}
  \lambda &\; \text{new} &\; \text{number of nonzero entries in coefficient vector for each new identity}
  \\
  \midrule
  7  &\;  2 &\;   
  4{,}4 
  \\ 
  61 &\;  9 &\;   
  3{,}6{,}6{,}11{,}12{,}14{,}14{,}16{,}17 
  \\ 
  52 &\; 17 &\;   
  8{,}18{,}18{,}19{,}19{,}21{,}30{,}30{,}32{,}32{,}32{,}33{,}33{,}34{,}38{,}43{,}44 
  \\
  51^2 &\; 20 &\;   
  6{,}7{,}13{,}13{,}14{,}20{,}20{,}20{,}22{,}24{,}26{,}26{,}27{,}28{,}29{,}30{,}38{,}41{,}45{,}51 
  \\ 
  43 &\; 15 &\; 
  9{,}10{,}16{,}16{,}20{,}21{,}21{,}23{,}27{,}28{,}34{,}35{,}41{,}43{,}45 
  \\
  421 &\; 20 &\;  
  15{,}27{,}31{,}33{,}35{,}41{,}42{,}48{,}51{,}51{,}72{,}90{,}102{,}104{,}119{,}121{,}124{,}128{,}128{,}129 
  \\
  3^21 &\; 11 &\;   
  9{,}10{,}13{,}13{,}36{,}36{,}40{,}48{,}59{,}71{,}73
  \end{array}
  }
  \]
  \caption{Nonzero matrix entries for new JJ identities in degree 7}
  \label{JJnew7info}
  \end{figure}


\section{Two Jordan products: nonlinear identities in degree 7} \label{sectionJJ7proof}

In this final section we explain in more detail our computational methods based on the expansion map 
\eqref{expansionmap} and the representation theory of the symmetric group \eqref{Sndecomp}.
We focus on the problem of finding explicit nonlinear forms of the 94 new irreducible multilinear
identities in degree 7 for two Jordan products in the free associative interchange algebra
(Lemma \ref{JJdegree7}).
Our main example will be partition $\lambda = 421$ which has 20 new identities (Figure \ref{JJ7mults}).
Together with its conjugate $\lambda^\ast = 321^2$, these two partitions correspond to the largest 
(35-dimensional) irreducible representations of $S_7$.

\subsection*{Combining the expansion matrix with representation theory}

We use the block matrix $B_n$ of equation \eqref{blockmatrix} to find all new identities in degree $n = 7$ 
for two Jordan products. 
However, we are no longer using all multilinear monomials obtained from all permutations of the variables
in the AA and JJ types, but rather the AA and JJ types together with the representation theory of the 
symmetric group.
There are 5040 permutations but only 35 standard tableaux for $\lambda = 421$, and so the matrices 
we use will be $5040/35 = 144$ times smaller than if we were using all multilinear monomials.
(This ratio is even greater for the smaller representations.)

In degree 7 there are 1806 AA types, 688 consequences of \eqref{interchange}, 616 JJ types, and 
the $S_7$-module $[421]$ has dimension 35.
Therefore, in the expansion matrix for $\lambda = 421$:
  \begin{itemize}
  \item
  The upper left block $\Xi$, which contains the consequences of the interchange identity, 
  has $688 \cdot 35 = 24080$ rows and $1806 \cdot 35 = 63210$ columns. 
  \item
  The lower left block $X$, which contains the expansions of the JJ monomials, 
  has $616 \cdot 35 = 21560$ rows and $1806 \cdot 35 = 63210$ columns. 
  \item
  The lower right identity matrix has $616 \cdot 35 = 21560$ rows and columns.
  \item
  The upper right zero matrix has $688 \cdot 35 = 24080$ rows, $616 \cdot 35 = 21560$ columns.
  \end{itemize}
Altogether, this expansion matrix has 45640 rows and 84770 columns.

To compute efficiently with a matrix of this size (3868902800 entries, almost 4 gigabytes at one 
byte per entry) we cannot use rational arithmetic, so we must use modular arithmetic followed by 
rational reconstruction.
Fortunately, we were able to complete these calculations using only single primes, $p = 101$ and
$p = 1000003$; we did not need the Chinese remainder theorem.

We compute the RCF of the expansion matrix and determine the rows whose leading 1s lie in 
the right side of the matrix, columns $j > 63210$.
There are 17752 such rows (``all'' in Figure \ref{JJ7mults}), which form a block of size 
$11752 \times 21560$, ignoring the zeros in the left side of the matrix ($j \le 63210$).
The rows of this block represent all multilinear identities in partition 421 for two Jordan products 
in degree 7.

To determine which rows of the block represent new identities in degree 7, we compute the RCF of 
the matrix representing the consequences of the known identities in degree $\le 6$.
There are 1520 symmetries of the JJ types in degree 7: multilinear identities of the form $m - m' \equiv 0$, 
where $m$ is a JJ type with the identity permutation of the variables, and $m'$ is the same JJ type 
with a permutation of order 2: these are the consequences of commutativity of the Jordan products.
There are 2676 consequences in degree 7 of (the linearizations of) the two Jordan identities in degree 4, 
and 2212 consequences of the new JJ identities from degrees 4--6.
All these identities are linear combinations of the JJ types with various permutations of the variables, 
and so their components for $\lambda = 421$ can be represented by a matrix with $616 \cdot 35 = 21560$ columns.
The total number of these old identities is 6408, and so we would require $6408 \times 35 = 224280$ rows
to process them together.
But the rank can never be greater than the number of columns, so we save memory by creating a matrix 
with $21560 + 3500 = 25060$ rows, and processing the old identities 100 at a time in 65 groups.
At the end of this iteration, we have a matrix in RCF with rank 17732 (``old'' in Figure \ref{JJ7mults});
discarding zero rows gives size $17732 \times 21560$.

Comparing the results so far, we see that there are $17752 - 17732 = 20$ new JJ identities in degree 7 
(``new'' in Figure \ref{JJ7mults}).
The row space of the second matrix (old identities) is a subspace of the row space of the first matrix
(all identities).
The new identities are represented by the rows of the first matrix whose leading 1s occur in columns 
for which the second matrix has no leading 1.

For $\lambda = 421$ the simplest of the rows representing new identities has 15 nonzero entries 
(see Figure \ref{JJnew7info}).
Each column index $j$ can be written as $j = 35(t-1) + \ell$ where $t$ is the JJ type and 
$\ell$ is the column index in the corresponding $35 \times 35$ representation matrix.
For each of the 15 nonzero entries, Figure \ref{JJnew7example} displays the indices $j$, $t$, $\ell$ 
and the entry $c$ (modulo 101); only two JJ types occur, $t = 595$ and $t = 611$.
In this case, rational reconstruction is very easy: we use symmetric representatives modulo 101 to 
replace $\{ 100, 99, 97 \}$ by $\{ -1, -2, -4 \}$ respectively, and interpret the resulting coefficients 
as integers.
According to the decomposition of the group algebra $\mathbb{F} S_n$ in equation \eqref{Sndecomp}, 
each nonzero matrix entry represents a scalar multiple of an element $R^{-1}(E_{k\ell})$ in the
group algebra where $E_{k\ell}$ is a $35 \times 35$ matrix unit.
The elements $R^{-1}(E_{k\ell})$ are defined, in terms of the standard tableaux for $\lambda$, 
to be the product of two factors: the sum over all row permutations and the alternating sum over all 
column permutations.
For $\lambda = 421$ with $\lambda^\ast = 321^2$, each of the elements $R^{-1}(E_{k\ell})$
has $4!2!3!2! = 576$ terms.
So this very sparse matrix row represents a multilinear identity with $15 \cdot 576 = 8640$ terms.
But this multilinear identity has symmetries determined by the partition $\lambda = 421$, 
and we can use these symmetries to find an equivalent but much smaller nonlinear identity.

  \smallskip
  
  \begin{figure}[ht]
  \small
  \[
  \boxed{
  \begin{array}
  {c@{\;\;}|
  c@{\quad}c@{\quad}c@{\quad}c@{\quad}c@{\quad}c@{\quad}c@{\quad}c@{\quad}c@{\quad}c@{\quad}
  c@{\quad}c@{\quad}c@{\quad}c@{\quad}c@{\quad}}
  j & 20792 & 20794 & 20797 & 20798 & 20815 & 20816 & 20819 & 20820 & 20821 & 20822 \\
  t & 595 & 595 & 595 & 595 & 595 & 595 & 595 & 595 & 595 & 595 \\
  \ell & 2 & 4 & 7 & 8 & 25 & 26 & 29 & 30 & 31 & 32 \\
  c & 1 & 1 & 99 & 1 & 2 & 97 & 2 & 97 & 1 & 1 \\ 
  \midrule
  j & 21354 & 21357 & 21372 & 21379 & 21380 \\
  t & 611 & 611 & 611 & 611 & 611 \\
  \ell & 4 & 7 & 22 & 29 & 30 \\
  c & 100 & 1 & 100 & 4 & 99 
  \end{array}
  }
  \]
  \caption{Matrix row representing simplest new identity for $\lambda = 421$}
  \label{JJnew7example}
  \end{figure}

\subsection*{Combining the expansion matrix with nonlinear identities}

The nonzero matrix entries in the rows representing the 20 new JJ identities for $\lambda = 421$
belong to columns corresponding to only 67 JJ types (out of the total of 616);
we call these the JJ subtypes.
The expansions of the JJ monomials for the JJ subtypes contain terms which involve only 374 AA types 
(out of the total of 1806).
However, to this subset of AA types $t$ we must add those AA types $t'$ for which some consequence
of the interchange identity involves both $t$ and $t'$.
This closure process converges after two iterations, leaving us with 530 AA types which must be 
included in our computations; we call these the AA subtypes.
There are 416 interchange consequences (out of the total of 688) in which both terms belong to the 
AA subtypes.

To find the new nonlinear identities, we use the expansion matrix with monomials;
we are no longer using representation theory. 
To obtain a basis of the domain of the expansion map, we apply each of the 67 JJ subtypes to all 
105 permutations of the nonlinear monomial $m = a^4b^2c$ corresponding to $\lambda = 421$.
We normalize the resulting JJ monomials using commutativity of the Jordan products, and obtain a 
set of 2417 nonlinear JJ monomials.
Similarly, we apply the 530 AA subtypes to the permutations of $m$ to obtain a basis of 
the codomain of the expansion map, and convert the 416 interchange consequences to a nonlinear form.
The resulting expansion matrix $E$ has 
$416 \cdot 105 + 2417 = 46097$ rows and $530 \cdot 105 + 2417 = 58067$ columns,
and the same block structure as equation \eqref{blockmatrix}.

We compute $R = \mathrm{RCF}(E)$ using a large prime $p = 1000003$ to facilitate rational reconstruction.
The matrix $R$ has rank 43577; the first 42787 rows have leading 1s in the left side 
(columns $j \le 55650$), and the remaining 790 rows have leading 1s in the right side: 
those columns are labeled by the 2417 nonlinear JJ monomials.
We extract the lower right block of size $790 \times 2417$ whose rows represent all the JJ identities
in degree 7 which have a permutation of $m$ as the underlying variables in each term.
We find that this block contains only 685 distinct residues modulo $p$ (a very small subset).

To do rational reconstruction, we consider all denominators $d = 2^i 3^j 5^k 7^\ell$
with prime factors $\le  7$ (the degree of the identities); each exponent $i,j,k,\ell$ has 
lower bound 0 and a suitable upper bound (we used 9, 3, 1, 1 respectively).
The theoretical justification is that in the explicit form of isomorphism \eqref{Sndecomp},
all the denominators are divisors of $n!$.
For each value of $d$, we multiply the 685 residues by $d$ and reduce again using symmetric 
representatives modulo $p$, obtaining a subset of integers contained in the interval $[-L,U]$ 
where $-L$ and $U$ are the min and max of the representatives.
The smallest interval (minimum of $L+U$) is obtained for $d = 2^6 \cdot 3 = 192$;
in this case the representatives have absolute value $\le 6528$ (which is most likely
the smallest possible).

At this point, we assume that we have found the correct integer values of the coefficients of 
the nonlinear JJ identities.
If we have made a mistake here by choosing an incorrect value of $d$ then we expect that obvious 
errors will appear when we attempt to confirm the identities independently using representation theory.
To say the same thing another way: we assume that if we were able to perform this computation 
using rational arithmetic, then $d = 192$ would be the LCM of the denominators of the entries 
of $\mathrm{RCF}(E)$.

If the number of integer vectors is not too large, roughly $< 500$, then at this point we can 
apply the LLL algorithm for lattice basis reduction \cite{BP1} to reduce the size of the coefficients.
In any case, we then sort the integer vectors first by increasing number of nonzero components, 
then by increasing Euclidean norm, and finally by increasing maximum nonzero component (in absolute value).

The final step is to process these identities using representation theory, and for this there are
two options: (i) linearize the identities, which increases the number of terms by a factor of 
$4!2! = 48$, or (ii) use the linearization operator \cite{BP2} which reduces the linearization of
each monomial to a single matrix multiplication.
In either case, we use representation theory as before to determine which of the 790 identities 
in the sorted list are new, where as usual ``new'' means not a consequence of the known identities 
of lower degree, but now also not a consequence of the previous identities in the list.

Figure \ref{JJnew7firstidentities} displays the first new nonlinear JJ identity for each of the 7
partitions in degree 7 for which the module of new identities is non-trivial; to save space we omit
``$\equiv 0$'' in these identities.
The other new nonlinear identities are available in an ancillary file attached to the arXiv version 
of this paper; see \cite{BMarXiv} for details.


\section*{Acknowledgements}

We thank the referee for numerous helpful comments and especially for suggesting that we make available 
online the polynomial identities which were too large for publication in a journal article.
The new nonlinear identities of degree 7 in two cases (Lie-Jordan and Jordan-Jordan) are available in 
ancillary files attached to the arXiv version \cite{BMarXiv}.
Murray Bremner was supported by a Discovery Grant from the Natural Sciences and Engineering Research Council of 
Canada (NSERC).
Sara Madariaga was supported by a Postdoctoral Fellowship from the Pacific Institute for Mathematical Sciences
(PIMS).

\end{document}